\documentclass[reqno,tbtags,a4paper,12pt]{amsart}

\usepackage{amsmath,amssymb,amsthm,amsfonts,amscd,array,graphicx}

\usepackage[in]{fullpage}

\newcommand{\ba}{\mathbf{a}}
\newcommand{\bb}{\mathbf{b}}
\newcommand{\bc}{\mathbf{c}}
\newcommand{\bd}{\mathbf{d}}
\newcommand{\be}{\mathbf{e}}
\newcommand{\bn}{\mathbf{n}}
\newcommand{\bm}{\mathbf{m}}

\newcommand{\bs}{\mathbf{s}}
\newcommand{\bv}{\mathbf{v}}
\newcommand{\bx}{\mathbf{x}}
\newcommand{\by}{\mathbf{y}}

\newcommand{\blambda}{\boldsymbol{\lambda}}

\newcommand{\sign}{\mathop{\mathrm{sign}}\nolimits}
\newcommand{\dist}{\mathop{\mathrm{dist}}\nolimits}
\newcommand{\spa}{\mathop{\mathrm{span}}\nolimits}

\newcommand{\R}{\mathbb{R}}
\newcommand{\CH}{\mathcal{H}}
\newcommand{\RP}{\mathbb{R}\mathrm{P}}
\newcommand{\E}{\mathbb{E}}

\newcommand{\Z}{\mathbb{Z}}

\newcommand{\X}{\mathbb{X}}
\newcommand{\V}{\mathbb{V}}
\newcommand{\CF}{\mathcal{F}}
\newcommand{\bS}{\mathbb{S}}

\newcommand{\CB}{\mathcal{B}}
\newcommand{\hCB}{\widehat{\mathcal{B}}}

\newcommand{\CV}{\mathcal{V}}

\newtheorem{theorem}{Theorem}[section]
\newtheorem{cor}[theorem]{Corollary}
\newtheorem{lem}[theorem]{Lemma}
\newtheorem{conjecture}[theorem]{Conjecture}
\theoremstyle{definition}
\newtheorem{defin}[theorem]{Definition}
\newtheorem{remark}[theorem]{Remark}

\author{Alexander~A.~Gaifullin}

\thanks{The work is partially supported by RFBR (projects 13-01-12469 and 14-01-00537), by a grant of the President of the Russian Federation (project MD-2969.2014.1), and by a grant from Dmitri Zimin's ``Dynasty'' foundation.}

\title{Embedded flexible spherical cross-polytopes with non-constant volumes}

\address{Steklov Mathematical Institute, Moscow, Russia\newline
${}$\hspace{4.3mm}Lomonosov Moscow State University, Moscow, Russia\newline 
${}$\hspace{4.3mm}Institute for Information Transmission Problems (Kharkevich Institute), Moscow, Russia}

\email{agaif@mi.ras.ru}

\begin{document}

\begin{abstract}
We construct examples of embedded flexible cross-polytopes in the spheres of all dimensions. These examples are interesting from two points of view. First, in dimensions~$4$ and higher, they are the first examples of embedded flexible polyhedra. Notice that, unlike in the spheres, in the Euclidean spaces and the Lobachevsky spaces of dimensions~$4$ and higher, still no example of an embedded flexible polyhedron is known. Second, we show that the volumes of the constructed flexible cross-polytopes are non-constant during the flexion. Hence these cross-polytopes give counterexamples to the Bellows Conjecture for spherical polyhedra. Earlier a counterexample to this conjecture was built only  in dimension~$3$ (Alexandrov, 1997), and was not embedded.  For flexible polyhedra in spheres we suggest a weakening of the Bellows Conjecture, which we call the \textit{Modified Bellows Conjecture}. We show that this conjecture holds for all flexible cross-polytopes of the simplest type among which there are our counterexamples to the usual Bellows Conjecture. By the way, we obtain several geometric results on flexible cross-polytopes of the simplest type. In particular, we write relations on the volumes of their faces of codimensions~$1$ and~$2$.
\end{abstract}


\maketitle

\begin{flushright}
\textit{To Nicolai Petrovich Dolbilin\\ on the occasion of his  seventieth birthday}
\end{flushright}

\section{Introduction}

Let $\X^n$ be one of the three $n$-dimensional spaces of constant curvature, that is,  the Euclidean space~$\E^n$ or the sphere~$\bS^n$ or the Lobachevsky space~$\Lambda^n$. For convenience, we shall always normalize metrics on the sphere~$\bS^n$ and on the Lobachevsky space~$\Lambda^n$ so that their curvatures are equal to~$1$ and~$-1$ respectively. For consistency of terminology,  great spheres in~$\bS^n$ will often be called planes. Spheres in~$\bS^n$ that are not great spheres will be called small spheres.

A \textit{flexible polyhedron\/} in~$\X^n$ is a closed connected  $(n-1)$-dimensional polyhedral surface~$P$ in~$\X^n$ that admits a continuous deformation~$P_u$ such that every face of~$P_u$ remains isometric to itself during the deformation. The surface~$P_u$ is allowed to be self-intersecting. However,  \textit{non-self-intersecting\/} (or \textit{embedded\/}) polyhedra are of a special interest. A precise definition will be given in Section~\ref{section_defin}. 

First flexible polyhedra, namely, flexible octahedra in~$\E^3$ were constructed by Bri\-card~\cite{Bri97}. Moreover, Bricard classified all flexible octahedra in~$\E^3$. In particular, he proved that all they are self-intersecting. The first example of an embedded flexible polyhedron in~$\E^3$ was constructed by Connelly~\cite{Con77}. Individual examples of flexible polyhedra in spaces $\E^4$, $\bS^3$ and~$\Lambda^3$ were constructed by Walz and Stachel, see~\cite{Sta00}, \cite{Sta06}. In a recent paper~\cite{Gai14a}, the author managed to generalize Bricard's results to all spaces of constant curvature of arbitrary dimensions, that is, to construct and to classify flexible cross-polytopes in all spaces~$\E^n$, $\bS^n$, and~$\Lambda^n$. Here and further an $n$-dimensional \textit{cross-polytope\/} is an arbitrary polyhedron of the combinatorial type of the regular cross-polytope, i.\,e., of the regular polytope dual to the $n$-dimensional cube. In particular, a two-dimensional cross-polytope is a quadrangle, and a three-dimensional cross-polytope is an octahedron.

Notice that there exists a very simple construction that allows to build a flexible polyhedron in~$\bS^n$ from every flexible polyhedron in~$\bS^{n-1}$: To do this one just need to take the bipyramid (the suspension) with vertices at the poles of~$\bS^n$ over the given flexible polyhedron lying in the equatorial great sphere $\bS^{n-1}\subset\bS^n$. Since any non-degenerate polygon in~$\bS^2$ with at least four sides is flexible, iterating the above construction, we can easily obtain many examples of flexible polyhedra in~$\bS^n$, including embedded.  Examples of such kind are not interesting. Therefore, it seems to be a right problem to study flexible polyhedra in the open hemisphere~$\bS^n_+\subset\bS^n$.   

Until now, no example of an embedded flexible polyhedron in~$\E^n$, $\bS^n_+$, or~$\Lambda^n$ was known for $n\ge 4$. In the present paper we shall show that there exist embedded flexible cross-polytopes  in the open hemispheres~$\bS^n_+$  of all dimensions. Moreover, we shall prove the following theorem.

\begin{theorem}\label{theorem_main}
For every $n\ge 2$, there exists a flexible cross-polytope $P_u,$ $u\in\overline{\R}=\R\cup\{\infty\},$ in the sphere~$\bS^n$ possessing the following properties:
\begin{enumerate}
\item There is a\/ $\delta>0$ such that the cross-polytope~$P_u$ is embedded whenever $u\in(-\delta,\delta).$
\item $P_0$ is the equatorial great sphere $\bS^{n-1}\subset\bS^n$ with a decomposition into simplices combinatorially equivalent to the boundary of the $n$-dimensional cross-polytope.
\item  $P_u$ is contained in the upper hemisphere~$\bS^n_+$ whenever~$u>0$, and is contained in the lower hemisphere~$\bS^n_-$ whenever $u<0$.
\item $P_{\infty}$ is contained in the equatorial great sphere $\bS^{n-1},$ but is not embedded.
\end{enumerate} 
\end{theorem}

This theorem will be proved in Section~\ref{section_embed}.

As it have been mentioned above, the classification of all flexible cross-polytopes in spaces of constant curvature was obtained by the author~~\cite{Gai14a}. However, this classification was given in algebraic terms. So the question of which of the constructed flexible cross-polytopes are self-intersecting and which are embedded is non-trivial and was not considered in~\cite{Gai14a}.  The proof of Theorem~\ref{theorem_main} reduces to showing that some of the flexible cross-polytopes constructed in~\cite{Gai14a} for an appropriate choice of parameters possess the required properties~\mbox{(i)--(iv)}.  We shall show that such cross-polytopes can be found in the class of  \textit{flexible cross-polytopes of the simplest type\/} that were constructed in Section~5 of~\cite{Gai14a} as multi-dimensional generalizations  of Bricard's flexible octahedra of the third type, which are also called  \textit{skew flexible octahedra\/}. The simplest type of flexible cross-polytopes  takes a special place because the flexions of cross-polytopes of this type admit a rational parametrization, while all other types of flexible cross-polytopes have an elliptic parametrization that degenerates to a rational parametrization only for some special values of edge lengths.

One of important problems in the theory of flexible polyhedra is the problem on their volumes related to the so-called \textit{Bellows Conjecture}. This conjecture, which was suggested by Connelly~\cite{Con78} and was proved by Sabitov~\cite{Sab96}--\cite{Sab98b} in 1996, claims that the volume of any flexible polyhedron is constant during the flexion. Since under a polyhedron we mean a polyhedral surface, we need to specify that under  the volume of a polyhedron we mean the volume of the region bounded by this polyhedral surface. If a polyhedron is self-intersecting the usual volume should be replaced by a so-called  \textit{generalized volume} whose precise definition will be given in Section~\ref{section_volume}. An alternative proof of the Bellows Conjecture was given in~\cite{CSW97}. A survey of these proofs and related results and problems can be found in~\cite{Sab11}. 

A multi-dimensional generalization of the Bellows Conjecture, that is, the assertion of the constancy of the volume of an arbitrary flexible polyhedron in~$\E^n$, $n\ge 4$, was proved by the author~\cite{Gai11}, \cite{Gai12}. The question naturally arises if the analogue of the Bellows Conjecture holds in the spheres~$\bS^n$ and in the Lobachevsky spaces~$\Lambda^n$, $n\ge 3$. More precisely, we again should replace the spheres~$\bS^n$ with the open hemispheres~$\bS^n_+$, since in~$\bS^n$ the question is trivial. Indeed, flexible polyhedra in~$\bS^n$ with non-constant volumes can be obtained by taking iterated bipyramids over flexible spherical  polygons with non-constant areas, as was described above. In 1997 Alexandrov~\cite{Ale97} constructed an example of a flexible self-intersecting polyhedron with non-constant volume in the open hemisphere~$\bS^3_+$. The combinatorial type of this polyhedron is the bipyramid over the hexagon. Theorem~\ref{theorem_main} easily yields the following result.

\begin{cor}\label{cor_main}
The volume of the flexible cross-polytope~$P_u$ in Theorem~\ref{theorem_main} is non-constant on any arbitrarily small interval of parameters $(0,u_0)$. Thus, for every~$n\ge 2,$ there exists an embedded flexible cross-polytope with non-constant volume in\/~$\bS^n_+$.
\end{cor}

\begin{proof}
For $u$ small enough, the surface~$P_u$ divides the sphere~$\bS^n$ into two parts.  
To speak on the volume of the cross-polytope~$P_u$, we should agree the volume of which of these two parts we consider. The sum of these two volumes is constant and is equal to the volume~$\sigma_n$ of the sphere~$\bS^n$. Since we are interested only in the question on the non-constancy of the volume, it is completely irrelevant for us which of the two volumes to consider. To be specific, we shall agree that the volume $V(P_u)$ of the polyhedron~$P_u$ is the volume of the part that contain the northern pole. Then $V(P_0)=\sigma_n/2$ and $V(P_u)<\sigma_n/2$ whenever $0<u<\delta$, since the polyhedron~$P_u$ is embedded and is contained in~$\bS^n_+$. Hence the function~$V(P_u)$ is non-constant in any neighborhood of zero.
\end{proof}   

Recently the author~\cite{Gai14b} has proved the Bellows Conjecture for flexible polyhedra in odd-dimensional Lobachevsky spaces. The question of whether the Bellows Conjecture is true for flexible polyhedra in even-dimensional spaces remains open. It is well-known that the volumes of polyhedra in the sphere~$\bS^n$ and in the Lobachevsky space~$\Lambda^n$ are closely related to each other. For instance, this occurs in the fact that the functions expressing the volumes of simplices in~$\bS^n$ and in~$\Lambda^n$ from their dihedral angles are obtained from each other (up to a multiplicative constant) by an appropriate analytic continuation~\cite{Cox35}, \cite{Aom77}, see also~\cite{AVS88}. 
Hence the fact that the Bellows Conjecture is true in odd-dimensional Lobachevsky spaces makes rather plausible the assumption that certain proper analogue of the Bellows Conjecture in spheres still should be true. To formulate this conjecture, which we shall refer to as the \textit{Modified Bellows Conjecture,} we shall need a special operation on polyhedra in~$\bS^n$.

First of all, we note that studying flexible polyhedra we can restrict ourselves to studying  only\textit{simplicial\/} flexible polyhedra, i.\,e., such that all their faces are simplices. Indeed, for an arbitrary flexible polyhedron we can decompose its faces into simplices, possibly, adding new vertices. Then the obtained polyhedron will again be flexible. Notice that, for simplicial polyhedra, a deformation preserving the combinatorial type is a flexion if and only if all edge lengths are constant during this deformation. Let $P_u$ be an arbitrary simplicial flexible polyhedron in~$\bS^n$, and let $\ba(u)$ be a vertex of it. Denote by~$-\ba(u)$ the point of~$\bS^n$ antipodal to the point~$\ba(u)$. Consider a new flexible polyhedron~$\widetilde{P}_u$ of the same combinatorial type as~$P_u$ such that all vertices of~$P_u$ except for~$\ba(u)$ and all faces of~$P_u$ not containing~$\ba(u)$ remain vertices and faces of~$\widetilde{P}_u$ respectively, the vertex~$\ba(u)$ of~$P_u$ is replaced by the vertex~$-\ba(u)$ of~$\widetilde{P}_u$, and every face $[\ba(u)\bb_1(u)\ldots\bb_k(u)]$ of~$P_u$ is replaced by the face  $[(-\ba(u))\bb_1(u)\ldots\bb_k(u)]$ of~$\widetilde{P}_u$. We shall say that the flexible polyhedron~$\widetilde{P}_u$ is obtained from the flexible polyhedron~$P_u$ by \textit{replacing the vertex\/~$\ba(u)$ by its antipode.} It is easy to see that the replacements of two different vertices of~$P_u$ by their antipodes commute.

\begin{conjecture}[Modified Bellows Conjecture]\label{con_mod}
Let $P_u$ be an arbitrary flexible simplicial polyhedron in~$\bS^n$. Then we can replace some vertices of~$P_u$ by their antipodes so that the generalized volume of the obtained flexible polyhedron will remain constant during the flexion.
\end{conjecture}

As the evidence of the plausibility of this conjecture, we shall show that it is true for all flexible cross-polytopes of the simplest type among which, by Corollary~\ref{cor_main}, there are counterexamples to the usual Bellows Conjecture. The proof will be based on an explicit calculation of the volumes of the flexible cross-polytopes of the simplest type (Theorem~\ref{theorem_vol}).  During this calculation, we shall obtain a series of results on the geometry of flexible cross-polytopes of the simplest type in~$\X^n=\E^n$, $\bS^n$, and~$\Lambda^n$, which, we believe, are of independent interest. In Section~\ref{section_dihedral} we derive formulae for the dihedral angles of flexible cross-polytopes of the simplest type. Each flexible cross-polytope of the simplest type~$P_u$ is flat (i.\,e., is contained in a hyperplane  $\X^{n-1}\subset\X^n$) for the two values $u=0$ and~$u=\infty$ of the parameter. In Section~\ref{section_flat} we prove that  flat cross-polytopes~$P_0$ and~$P_{\infty}$  have certain surprising geometric properties. Namely, the $(n-1)$-dimensional cross-polytopes obtained from~$P_0$ (or~$P_{\infty}$) by deleting different pairs of opposite vertices are either all circumscribed about concentric spheres or satisfy certain other similar properties (Theorem~\ref{theorem_opisan}). For flexible octahedra in the three-dimensional Euclidean space this result was obtained by Bennett~\cite{Ben12}. In Section~\ref{section_volume} we derive linear relations on the volumes of  $(n-1)$-dimensional and $(n-2)$-dimensional faces of flexible cross-polytopes of the simplest type. In the case of the three-dimensional Euclidean space, these relations turn to relations on the areas of faces and the lengths of edges of skew flexible octahedra, which were known to Bricard~\cite{Bri97}.

\section{Definition of flexible polyhedra}\label{section_defin}

\begin{defin}\label{defin_polyh}
A finite simplicial complex~$K$ is called a $k$-dimensional \textit{pseudo-manifold\/} if 
\begin{enumerate}
\item every simplex of~$K$ is contained in a $k$-dimensional simplex of~$K$,
\item every $(k-1)$-dimensional simplex of~$K$ is contained in exactly two $k$-dimensional simplices of~$K$,
\item $K$ is \textit{strongly connected,\/} i.\,e., any two $k$-dimensional simplices of~$K$ can be connected by a finite sequence of $k$-dimensional simplices such that any two consecutive simplices in this sequence have a common  $(k-1)$-dimensional face.
\end{enumerate}
We say that a pseudo-manifold~$K$ is \textit{oriented\/} if all its $k$-dimensional simplices are endowed with orientations such that, for any $(k-1)$-simplex~$\tau$ of~$K$, the orientations induced on~$\tau$ by the chosen orientations of the two $k$-dimensional simplices containing~$\tau$ are opposite to each other.
\end{defin}

\begin{defin}\label{defin_nondeg}
Let $K$ be an oriented $(n-1)$-dimensional pseudo-manifold. A \textit{non-degenerate polyhedron\/} (or \textit{polyhedral surface\/})  of combinatorial type~$K$ in~$\X^n$ is a mapping $P\colon K\to\X^n$ such that
\begin{enumerate}
\item For each simplex  $[v_0\ldots v_l]$ of~$K$ the points $P(v_0),\ldots,P(v_l)$ are independent, i.\,e., do not lie in an $(l-1)$-dimensional plane in~$\X^n$,  and the restriction of~$P$ to the simplex $[v_0\ldots v_l]$ is a homeomorphism onto the convex hull of the points $P(v_0),\ldots,P(v_l)$.
\item $K$ cannot be decomposed into the union of two subcomplexes~$K_1$ and~$K_2$ such that $\dim K_1=\dim K_2=n-1$ and the set $P(K_1\cap K_2)$ is contained in an $(n-2)$-dimensional plane in~$\X^n$.
\end{enumerate}
A polyhedron $P\colon K\to\X^n$ is called \textit{embedded} if~$P$ is an embedding, and is called  \textit{self-intersecting\/} otherwise. The number~$n$ is called the \textit{dimension\/} of the polyhedron~$P$. The images of simplices of~$K$ under the mapping~$P$ are called 
\textit{faces\/} of the polyhedron. Faces of codimension~$1$, i.\,e., of dimension~$n-1$, are called  \textit{facets}. We agree that the whole polyhedron is not a face of itself.
\end{defin}

It is completely irrelevant which homeomorphism is used to map a simplex $[v_0\ldots v_l]$ onto the convex hull of the points $P(v_0),\ldots,P(v_l)$. This means that we do not distinguish between polyhedra~$P$ and~$P'$ of the same combinatorial type~$K$ such that $P(v)=P'(v)$ for all vertices~$v$ of~$K$.

\begin{defin}
A continuous family of mappings $P_u\colon K\to\X^n$ is called a \textit{non-degenerate flexible polyhedron\/} if:
\begin{enumerate}
\item For all but a finite number of~$u$,  $P_u$ is a non-degenerate polyhedron of combinatorial type~$K$. 
\item The lengths of all edges $[P_u(v_1)P_u(v_2)]$ of the polyhedron~$P_u$ are constant as~$u$ varies. 
\item For any two sufficiently close to each other $u_1\ne u_2$, the polyhedra~$P_{u_1}$ and~$P_{u_2}$ are not congruent to each other.
\end{enumerate} 
\end{defin}

Let us give an example showing why we need condition~(ii)  in Definition~\ref{defin_nondeg}. Consider the two-dimensional pseudo-manifold~$K$ with $7$ vertices $\ba_1$, $\ba_2$, $\bb_1$, $\bb_2$, $\bc_1$, $\bc_2$, $\bd$, and $10$ two-dimensional simplices $[\ba_1\ba_2\bb_j]$,  $[\ba_i\bb_j\bc_j]$, $[\ba_i\bc_j\bd]$, $i,j\in\{1,2\}$. Consider the polyhedron $P\colon K\to \E^3$  with the vertices $$P(\ba_i)=((-1)^i,0,0),\quad P(\bb_i)=(0,(-1)^i,0),\quad P(\bc_i)=(0,0,(-1)^i),\quad P(\bd)=(0,0,0).$$
Geometrically this polyhedron is two tetrahedra with a common edge. Naturally, it admits flexions consisting in rotations of these tetrahedra around their common edge. Condition~(ii) is introduced to exclude such examples. Notice that, for an embedded polyhedron, this condition always holds automatically. In the sequel, we always mean that all flexible polyhedra under consideration are non-degenerate without mentioning this explicitly.

It is easy to check that, for flexible cross-polytopes, the condition of non-degeneracy is equivalent to the requirement that none of the dihedral angles is either identically~$0$ or identically~$\pi$ during the flexion. This requirement was imposed in~\cite{Gai14a} in the classification of flexible cross-polytopes, and was called  \textit{essentiality\/}. 

Consider a face $G$ of dimension $k<n-1$ of a polyhedron $P\colon K\to\X^n$. Take a point~$\bx$ in the relative interior of~$G$. (As usually, we shall conveniently agree that the relative interior of a vertex is this vertex itself.) Denote by~$\bS^{n-k-1}_{\bx}$ the sphere of unit vectors orthogonal to~$G$ in the tangent space~$T_{\bx}\X^n$. For each face~$F\supset G$ of~$P$, the cone of tangent vectors to~$F$ at~$\bx$ cuts out in the sphere~$\bS^{n-k-1}_{\bx}$ a spherical simplex of dimension $\dim F-k-1$. These simplices for all faces $F\supset G$ constitute an $(n-k-1)$-dimensional spherical polyhedron in~$\bS^{n-k-1}_{\bx}$, which we shall denote by~$L(G,P)$ and shall call the  \textit{link\/} of~$G$ in~$P$. Up to an isometry, the link of~$G$ is independent of the choice of the point~$\bx$. If $P_u$ is a flexible polyhedron in~$\X^n$, then all faces of~$P_u$ remain congruent to themselves during the flexion. Hence, for each pair of faces $F\supset G$ the spherical simplex corresponding to it also remains congruent to itself. Therefore, $L(G,P_u)$ is a flexible polyhedron in~$\bS^{n-k-1}$.

For an embedded polyhedron in the Euclidean space or in the Lobachevsky space, one can naturally define its \textit{interior dihedral angles\/} at faces of codimension~$2$. For an embedded polyhedron in the sphere this definition is not unambiguous, since we need to specify which of the two components of the complement of the polyhedral surface is called the interior. For a self-intersecting polyhedron, the concept of an interior dihedral angles looses its sense. Hence, for arbitrary polyhedra,  the right object is \textit{oriented dihedral angles\/} which can be defined in the following way. First of all, we need to fix an orientation of the pseudo-manifold~$K$ and an orientation of the space~$\X^n$. For a facet~$\Delta_i$ of a non-degenerate polyhedron $P\colon K\to\X^n$,  the \textit{unit exterior normal vector\/} to it at its point~$\bx$ is, by definition, the unit vector $\bm_i\in T_{\bx}\Delta_i$ orthogonal to~$\Delta_i$ such that the product of the direction of~$\bm_i$ by the orientation of the facet~$\Delta_i$ induced by the given orientation of~$K$ yields the positive orientation of~$\X^n$. Let  $F$ be an $(n-2)$-dimensional face of~$P$, and let $\Delta_1$ and~$\Delta_2$ be the two facets containing~$F$. Choose an arbitrary point $\bx\in F$. Let $\bm_1$ and~$\bm_2$ be the unit exterior normal vectors to the facets~$\Delta_1$ and~$\Delta_2$ respectively at the point~$\bx$. For $i=1,2$, we denote by~$\bn_i$ the unit interior normal vector to the face~$F$   of the simplex~$\Delta_i$ at the point~$\bx$, that is, the unit vector in~$T_{\bx}\Delta_i$ orthogonal to the simplex~$F$ and pointing inside~$\Delta_i$.  Choose a positive rotation direction  around the face~$F$ such that the vector  $\bn_1$ is obtained from the vector~$\bm_1$ by the rotation by the angle~$\pi/2$ in the positive direction. Now, we denote by~$\psi_F$ the rotation angle of the vector~$\bn_1$ to the vector~$\bn_2$ in this positive direction. This angle is well defined up to~$2\pi q$, $q\in\Z$. We shall regard this angle as the element of the group~$\R/(2\pi\Z)$. It is easy to show that the angle~$\psi_F$ is independent of the choice of the point~$\bx$ and of which of the two facets containing~$F$ is denoted by~$\Delta_1$. This angle will be called the  \textit{oriented dihedral angle\/} of the polyhedron~$P$ at the face~$F$.

\section{Flexible cross-polytopes of the simplest type}\label{section_simplest}
In this section we give the construction of flexible cross-polytopes of the simplest type obtained by the author in~\cite[Section 5]{Gai14a}. We shall always identify the Euclidean space~$\E^n$ with the Euclidean vector space~$\R^{n}$, the sphere~$\bS^n$ with the unit sphere in the Euclidean vector space~$\R^{n+1}$, and the Lobachevsky space~$\Lambda^n$ with the half of the two-pole hyperboloid $\langle \bx,\bx\rangle=-1$, $x_0>0$, in the pseudo-Euclidean space~$\R^{n,1}$ with the scalar product
$$
\langle\bx,\by\rangle=-x_0y_0+x_1y_1+\ldots+x_ny_n.
$$
To unify the notation, we shall denote by~$\V$ the spaces~$\R^n$, $\R^{n+1}$, and~$\R^{n,1}$  in the cases $\X^n=\E^n$, $\bS^n$, and~$\Lambda^n$ respectively.

We consider the simplicial complex~$K_n$ with $2n$ vertices~$\ba_1,\ldots,\ba_n$, $\bb_1,\ldots,\bb_n$, such that edges are spanned by all pairs of vertices except for the pairs~$(\ba_i,\bb_i)$, $i=1,\ldots,n$, and simplices are spanned by all sets of vertices that are pairwise joint by edges. The simplicial complex~$K_n$ is isomorphic to the boundary of the regular $n$-dimensional cross-polytope. Any polyhedron of combinatorial type~$K_n$ will be called a \textit{cross-polytope\/}. With some abuse of notation, we shall denote the vertices of~$P$ that are the images of the vertices~$\ba_i$ and~$\bb_i$ of~$K_n$ again by~$\ba_i$ and~$\bb_i$ respectively. We choose the orientation of the simplicial complex~$K_n$ such that the simplex~$[\ba_1\ldots\ba_n]$ is positively oriented. Denote the set~$\{1,\ldots,n\}$ by~$[n]$. For any disjoint subsets~$I,J\subset[n]$, we denote by~$\Delta_{I,J}$ the face of the cross-polytope~$P$ spanned by all vertices~$\ba_i$, $i\in I$, and all vertices~$\bb_j$, $j\in J$. Obviously, $\dim\Delta_{I,J}=|I|+|J|-1$. The fact that the subsets~$I$ and~$J$ in notation like~$\Delta_{I,J}$ are disjoint is always implied and is not stated explicitly.

A flexible cross-polytope~$P_u$ of the simplest type in~$\X^n$ corresponds to a pair $(G,\blambda)$, where 
\begin{enumerate}
\item $G=(g_{ij})$ is a symmetric matrix of size $n\times n$ with units on the diagonal such that all its principal minors of sizes $2\times 2,$ $\ldots,$ $(n-1)\times(n-1)$ are strictly positive, and $\det G>0$, $\det G=0$, and $\det G<0$  in the cases $\X^n=\bS^n$, $\X^n=\E^n$, and $\X^n=\Lambda^n$ respecttively, 
\item $\blambda=(\lambda_1,\ldots,\lambda_n)$  is a row of non-zero real numbers such that $\lambda_i\ne\pm\lambda_j$ unless $i= j$. 
\end{enumerate}

If $\X^n=\bS^n$ or $\Lambda^n$, then the flexible cross-polytope~$P_u$ is built in the following way:

1. Choose vectors $\bn_1,\ldots,\bn_n\in\V$ with the Gram matrix~$G$, and a vector~$\bm\in\V$ orthogonal to them such that $\langle\bm,\bm\rangle=1$. This can be done in a unique way up to a (pseudo)orthogonal transformation of~$\V$.

2. Determine  elements of the matrix $H=(h_{ij})$  of size $n\times n$ by 
\begin{equation*}
h_{ij}=\frac{2\lambda_i(\lambda_ig_{ij}-\lambda_j)}{\lambda_i^2-\lambda_j^2}
\end{equation*}
whenever $i\ne j$, and put $h_{ii}=1$ for all $i$.

3. Take the basis $\bc_1,\ldots,\bc_n$ of the subspace $\spa(\bn_1,\ldots,\bn_n)$ dual to the basis~$\bn_1,\ldots,\bn_n$, and put
\begin{equation}\label{eq_bd}
\bd_i(u)=\sum_{j=1}^nh_{ij}\bc_j-\frac{2\lambda_i^2u^2}{\lambda_i^2u^2+1}\,\bn_i+\frac{2\lambda_iu}{\lambda_i^2u^2+1}\,\bm.
\end{equation}

4. Then the parametrization of the flexion of the cross-polytope is given by
\begin{equation}\label{eq_param_SL}
\ba_i(u)=\frac{s_i\bc_i}{|\bc_i|}\,,\qquad\bb_i(u)=\frac{s_i'\bd_i(u)}{|\bd_i(u)|}\,,
\end{equation}
where the signs~$s_i,s_i'=\pm1$ are chosen arbitrarily in the case $\X^n=\bS^n$, and are chosen so that the points~$\ba_i(u)$ and~$\bb_i(u)$ belong to the connected component~$\Lambda^n$ of the hyperboloid $\langle\bx,\bx\rangle=-1$ in the case $\X^n=\Lambda^n$. (In fact, a straightforward computation shows that the length of the vector~$\bd_i(u)$ is independent of~$u$.)

If $\X^n=\E^n$, then Steps 1 and 2 are the same as above, and Steps~3 and~4 are as follows:

3.  Take a hyperplane $\Pi\subset\E^n$ orthogonal to~$\bm$, and a simplex~$[\ba_1\ldots\ba_n]$  in~$\Pi$ such that the vectors $\bn_1,\ldots,\bn_n$ are orthogonal to the facets of~$[\ba_1\ldots\ba_n]$ opposite to the vertices $\ba_1,\ldots,\ba_n$ respectively. It is easy to see that such simplex is unique up to a homothety and a parallel translation. For each~$i$, take the length of the altitude of the simplex $[\ba_1\ldots\ba_n]$ drawn from the vertex~$\ba_i$, and multiply it by the sign~$s_i$ that is equal to~$+1$ whenever~$\bn_i$ is the interior normal vector to the corresponding facet of $[\ba_1\ldots\ba_n]$, and is equal to~$-1$ whenever~$\bn_i$ is the exterior normal vector to the corresponding facet of $[\ba_1\ldots\ba_n]$. Denote the obtained number by~$a_i$.

4. Then the parametrization of the flexion of the cross-polytope is given by
\begin{gather}\label{eq_param_E}
\ba_i(u)=\ba_i,\qquad \bb_i(u)=b_i\left(\sum_{j=1}^n\frac{h_{ij}\ba_j}{a_j}-\frac{2\lambda_i^2u^2}{\lambda_i^2u^2+1}\,\bn_i+\frac{2\lambda_iu}{\lambda_i^2u^2+1}\,\bm\right),\\ b_i=\left(\sum_{j=1}^n  \frac{h_{ij}}{a_j}\right)^{-1}.\label{eq_param_E2}
\end{gather}
In this case, we denote by~$s_i'$ the sign of the number~$b_i$, $i=1,\ldots,n$.
 
\begin{remark}\label{rem_all}
In the Euclidean and the spherical cases, the above construction yields a flexible cross-polytope if and  only if none of the denominators in the formulae written above vanishes, hence, for all pairs~$(G,\blambda)$ satisfying the above conditions~(i) and (ii) off some subset of positive codimension. For the Lobachevsky space the situation is somewhat more difficult. Namely, the cross-polytope is well defined only if the vectors~$\bd_i(u)$ computed by~\eqref{eq_bd} turn out to be time-like.
\end{remark} 
 
\begin{remark}
The fact that the formulae written above actually yield flexible polyhedra, i.\,e., that the lengths of all edges actually remain constant during the obtained deformations was proved in~\cite{Gai14a}. However, indeed, this fact can be checked immediately by a simple calculation without the usage of results of~\cite{Gai14a}. 
\end{remark}

\begin{remark}\label{rem_dihedral}
In~\cite{Gai14a} the author has also shown that, for flexible cross-polytopes of the simplest type described above, the dihedral angles adjacent to the face~$[\ba_1\ldots\ba_n]$ vary during the flexion so that the tangents of the halves of any two of them are either directly or inversely proportional to each other. Moreover, in the same paper it has been shown that, provided that $\X^n\ne \E^2$, this property is characteristic for flexible polyhedra of the simplest type. Namely, if the dihedral angle adjacent to a facet of a non-degenerate flexible cross-polytope vary so that the tangents of the halves of any two of them are either directly or inversely proportional to each other, then the flexion of this cross-polytope can be parametrized as indicated above. In the Euclidean plane~$\E^2$,  besides flexible cross-polytopes (quadrangles) of the simplest type, this characteristic property is also fulfilled for flexible parallelograms, which cannot be obtained by the construction described above, see Lemma~4.8 and Remark~4.9 in~\cite{Gai14a}.
\end{remark}

It is easy to see that, for each~$i$, the simultaneous changing signs of the numbers~$\lambda_i$, $s_i$, and $s_i'$, all matrix elements $g_{ij}=g_{ji}$ such that $j\ne i$, and the vectors~$\bm$ and~$\bn_i$ does not change the cross-polytope~$P_u$. Hence, without loss of generality, we may assume that all coefficients~$\lambda_i$ are positive. Besides, renumbering the vertices of the cross-polytope, we may achieve that  $0<\lambda_1<\cdots<\lambda_n$. In the sequel, we shall always assume that these inequalities are satisfied.

In the rest of this paper, $P_u$ is always a flexible cross-polytope of the simplest type, $(G,\blambda)$ is the corresponding data, and $\bs=(s_1,\ldots,s_n,s_1',\ldots,s_n')$ is the corresponding row of signs.

\section{Dihedral angles}\label{section_dihedral}

To provide that the oriented dihedral angles of the cross-polytope~$P_u$ are well defined, we need to choose the orientation of the space~$\X^n$. We shall conveniently choose this orientation so that the  vector~$\bm$ is the interior normal vector to the simplex~$[\ba_1\ldots\ba_n]$ if $s_1\ldots s_n=1$, and   is the exterior normal vector to the simplex~$[\ba_1\ldots\ba_n]$ if  $s_1\ldots s_n=-1$. The sign  $s_1\ldots s_n$ is introduced to provide that in the spherical case the orientation of the sphere~$\bS^n$ does not change if we replace some vertices of~$P_u$ with their antipodes.

Each $(n-2)$-dimensional face of~$P_u$ has the form $\Delta_{I,J}$, where  $|I|+|J|=n-1$. We put $\psi_{I,J}(u)=\psi_{\Delta_{I,J}}(u)$.

For each $k\in[n]$, we consider the set
$$
X_k=\big\{i\in [n]\mid \big((i<k)\wedge (s_is_i'=1)\big)\vee\big((i>k)\wedge (s_is_i'=-1)\big)\big\},
$$
where $\wedge$ and~$\vee$ denotes the logical ``and'' and ``or'' respectively.

\begin{lem}\label{lem_dihedral}
If either $n\ge 3$ or $n=2$ and $\X^2=\bS^2,$ then the dihedral angles of the cross-polytope~$P_u$ are given by
\begin{equation}\label{eq_dihedral}
\psi_{I,J}(u)=(-1)^{|J\cap X_k|}s_k\varphi_k(u)+\left\{
\begin{aligned}
&0 &&\text{if $s_ks_k'=1,$}\\
&\pi &&\text{if $s_ks_k'=-1,$}\\
\end{aligned}
\right.
\end{equation}
where $k$ is a unique element of the set $[n]\setminus(I\cup J),$ and  $\varphi_k(u)=2\arctan(\lambda_ku)$.
\end{lem}

\begin{proof}
First of all, we consider the dihedral angles $\psi_k(u)=\psi_{I\setminus\{k\},\emptyset}(u)$ adjacent to the facet~$\Delta_{[n],\emptyset}=[\ba_1\ldots\ba_n]$. For them, formula~\eqref{eq_dihedral} was substantially obtained in the author's paper~\cite{Gai14a}. Indeed, in this paper, evaluating the parametrization of the cross-polytopes of the simplest type, we have used the variables~$t_k$ that have been originally defined as the tangents of the halves of the dihedral angles adjacent to the facet $[\ba_1\ldots\ba_n]$. Then, to simplify the formulae, we have applied several special algebraic transformations, which we have called \textit{elementary reversions.} As a result of these transformations each of the variables~$t_k$ have been replaced with one of the values~$\pm t_k^{\pm 1}$. Afterwards, the new variables~$t_k$ have been parametrized by $t_k=\lambda_ku$. Thus, the results of~\cite{Gai14a} imply immediately that $\tan(\psi_k(u)/2)=\pm(\lambda_ku)^{\pm 1}$, i.\,e., that  $\psi_k(u)$ is one of the angles  $\pm\varphi_k(u)$, $\pm\varphi_k(u)+\pi$. (Recall that the angle~$\psi_k(u)$ is defined modulo~$2\pi\Z$.) It remains to show that the sign~$\pm$ at~$\varphi_k(u)$ is equal to~$s_k$, and the summand $\pi$ is present if and only if $s_ks_k'=-1$. The sign of the tangent of one half of the angle~$\psi_k(u)$ is equal to the sign of the sinus of the angle~$\psi_k(u)$, hence, is equal to the sign of the scalar product $\langle\bb_k(u),\bm\rangle$, which is equal to $s_k'\sign(\lambda_ku)$. Therefore, the dihedral angle~$\psi_k(u)$ equals either $s_k'\varphi_k(u)$ or $\pi-s_k'\varphi_k(u)$. Further, from formulae~\eqref{eq_bd}--\eqref{eq_param_E2}, which provide the parametrization for the flexion of the cross-polytope~$P_u$, one can easily deduce that $\dist_{\X^n}(\ba_k,\bb(0))<\dist_{\X^n}(\ba_k,\bb(\infty))$ whenever $s_ks_k'=1$ and $\dist_{\X^n}(\ba_k,\bb(0))>\dist_{\X^n}(\ba_k,\bb(\infty))$ whenever $s_ks_k'=-1$. Thus, if $s_ks_k'=1$, then $\psi_k(0)=0$ and $\psi_k(\infty)=\pi$, hence, $\psi_k(u)=s_k\varphi_k(u)$, and if $s_ks_k'=-1$, then $\psi_k(0)=\pi$ and $\psi_k(\infty)=0$, hence, $\psi_k(u)=s_k\varphi_k(u)+\pi$.

Now, let $U,W\subset[n]$ be subsets such that $U\cap W=\emptyset$ and $|U|+|W|=n-2$, and let $k$ and~$l$ be the two distinct elements of the set $[n]\setminus(U\cup W)$. Consider the formula
\begin{equation}\label{eq_neigh_dihedral}
\psi_{U,W\cup\{l\}}=\left\{
\begin{aligned}
-&\psi_{U\cup\{l\},W}&&\text{if $l\in X_k$,}\\
&\psi_{U\cup\{l\},W}&&\text{if $l\notin X_k$.}
\end{aligned}
\right.
\end{equation}
We shall proof formulae~\eqref{eq_dihedral} and~\eqref{eq_neigh_dihedral} by the simultaneous induction on~$|J|$ and~$|W|$ respectively. Formula~\eqref{eq_dihedral} has already been proved for $|J|=0$. First, we shall show that formula~\eqref{eq_dihedral} for all pairs~$(I,J)$ such that $|J|=p$ implies formula~\eqref{eq_neigh_dihedral} for all pairs~$(U,W)$ such that $|W|=p$. Second, we shall show that formula~\eqref{eq_neigh_dihedral} for all pairs~$(U,W)$ such that $|W|=p$ and formula~\eqref{eq_dihedral} for all pairs~$(I,J)$ such that $|J|=p$ imply formula~\eqref{eq_dihedral} for all pairs~$(I,J)$ such that $|J|=p+1$. As a result we shall complete the proofs of formulae~\eqref{eq_dihedral} and~\eqref{eq_neigh_dihedral}.

1. Suppose that formula~\eqref{eq_dihedral} holds true for all pairs~$(I,J)$ such that $|J|=p$. Consider an arbitrary pair~$(U,W)$ such that  $|W|=p$. If $n\ge 3$, then the link $L_{U,W}(u)=L(\Delta_{U,W},P_u)$ is a spherical quadrangle. We denote by~$A$, $B$, $C$, and~$D$ the vertices of this quadrangle corresponding to the $(n-2)$-dimensional faces~$\Delta_{U\cup\{l\},W}$, $\Delta_{U\cup\{k\},W}$, $\Delta_{U,W\cup\{l\}}$, and~$\Delta_{U,W\cup\{k\}}$ respectively.  Then $A$, $B$, $C$, and~$D$ are consecutive  vertices of the quadrangle~$L_{U,W}(u)$ (in the cyclic order). 

In the exceptional case of $n=2$ and $\X^2=\bS^2$, the condition $|U|+|W|=n-2=0$ implies immediately that $U=W=\emptyset$. Hence the face $\Delta_{U,W}=\Delta_{\emptyset,\emptyset}$ does not exist. Neverteless, we shall conveniently use the convention that the cross-polytope~$P_u$ has an additional empty face~$\Delta_{\emptyset,\emptyset}$, which is assigned formally the dimension~$-1$. The link of this face is, by definition, the cross-polytope~$P_u$ itself. Thus, in the case of  $n=2$ and $\X^2=\bS^2$, the link $L_{U,W}(u)=L_{\emptyset,\emptyset}(u)$ is again a spherical quadrangle with the vertices $A=\Delta_{\{l\},\emptyset}=\ba_l$, $B=\Delta_{\{k\},\emptyset}=\ba_k$, $C=\Delta_{\emptyset,\{l\}}=\bb_l$, and $D=\Delta_{\emptyset,\{k\}}=\bb_k$, where $(k,l)=(1,2)$ or $(2,1)$.

The oriented angles $\psi_A(u)$, $\psi_B(u)$, $\psi_C(u)$, and~$\psi_D(u)$ of the spherical quadrangle $L_{U,W}(u)$ at the vertices~$A$, $B$, $C$, and~$D$ respectively are equal to~$\psi_{U\cup\{l\},W}(u)$, $\psi_{U\cup\{k\},W}(u)$, $\psi_{U,W\cup\{l\}}(u)$, and~$\psi_{U,W\cup\{k\}}(u)$ respectively. By the inductive assumption, formula~\eqref{eq_dihedral} holds true for the pairs $(U\cup\{l\},W)$ and~$(U\cup\{k\},W)$. Hence, during the flexion of the quadrangle~$L_{U,W}(u)$, its angles at the vertices~$A$ and~$B$ vary in such a way that the tangents of their halves are either directly or inversely proportional to each other. Flexible spherical quadrangles with this property were completely described by Bricard~\cite[\S II]{Bri97} in his study of flexible octahedra in~$\E^3$. (Bricard considered tetrahedral angles in~$\E^3$ instead of spherical quadr\-ang\-les, but, obviously, these two objects are equivalent.) By the result of Bricard, for each such flexible spherical quadrangle, either its opposite sides are pairwise equal to each other or the sum of the lengths of each pair of its opposite sides is equal to~$\pi$. There exist three types of such quadrangles; they are shown in Fig.~\ref{fig_3polyg}. The vertices of these quadrangles are denoted by~$P$, $Q$, $R$, and~$S$ instead of~$A$, $B$, $C$, and~$D$, since they can be identified with~$A$, $B$, $C$, and~$D$ in different ways, which differ from each other by cyclic permutations. For quadrangles of the first and of the second types shown in Fig.~\ref{fig_3polyg}(a) and (b) respectively, we have $PQ=RS$ and $QR=SP$, and the tangents of the halves of neighbor angles are inversely proportional to each other during the flexion. Hence these cases occur when $s_ks_k'=-s_ls_l'$. For quadrangles of the third type shown in Fig.~\ref{fig_3polyg}(c), we have $PQ+RS=QR+SP=\pi$,  and the tangents of the halves of neighbor angles are directly proportional to each other during the flexion. Hence these cases occur when $s_ks_k'=s_ls_l'$.

\begin{figure}
\begin{center}
\includegraphics[scale=.4]{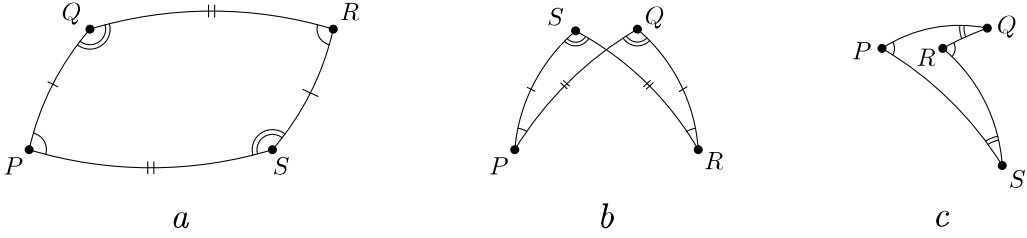}
\end{center}
\caption{Three types of flexible spherical quadrangles with either directly or inversely proportional to each other tangents of the half-angles}\label{fig_3polyg}
\end{figure}

Assume that the quadrangle $L_{U,W}(u)$ is of the first type.  Then
\begin{equation}\label{eq_psi+}
\psi_{U,W\cup\{l\}}(u)=\psi_C(u)=\psi_A(u)=\psi_{U\cup\{l\},W}(u).
\end{equation} 
Let us show that, in this case,  $l\notin X_k$.
Take $u>0$. If $s_ks_k'=1$ and $s_ls_l'=-1$, then formula~\eqref{eq_dihedral} for the pairs $(U\cup\{l\},W)$ and $(U\cup\{k\},W)$ implies that $\psi_A(u)=\pm\varphi_k(u)$ and $\psi_B(u)=\pm\varphi_l(u)+\pi$. Hence, the \textit{unoriented\/} interior angles of the quadrangle~$L_{U,W}(u)$ are equal to $\varphi_k(u)$ and $\pi-\varphi_l(u)$. Since the sum of the angles of a spherical quadrangle is greater than~$2\pi$, this implies tha $\varphi_k(u)> \varphi_l(u)$, therefore, $l<k$. Thus, $l\notin X_k$. If $s_ks_k'=-1$ and $s_ls_l'=1$, then the interior angles of the quadrangle $L_{U,W}(u)$ are equal to $\pi-\varphi_k(u)$ and $\varphi_l(u)$, hence, $l>k$, and again $l\notin X_k$.

If the quadrangle $L_{U,W}(u)$ is of the second type, then 
\begin{equation}\label{eq_psi-}
\psi_{U,W\cup\{l\}}(u)=\psi_C(u)=-\psi_A(u)=-\psi_{U\cup\{l\},W}(u).
\end{equation}
Similarly to the previous case, we easily deduce that $l\in X_k$ from the fact that the sum of the angles indicated by one and two arcs in Fig.~\ref{fig_3polyg}(b) is less than~$\pi$.

Assume that the quadrangle $L_{U,W}(u)$ is of the third type. If the vertex~$A$ is identified either with~$P$ or with~$R$, then equality~\eqref{eq_psi-} holds true, and if $A$ is identified either with~$Q$ or with~$S$, then equality~\eqref{eq_psi+} holds true. Similarly to the previous cases, since the angle indicated by one arc is greater than the angle indicated by two arcs in Fig.~\ref{fig_3polyg}(c), it follows easily that $l\in X_k$ whenever $A$ is either~$P$ or~$R$, and $l\notin X_k$ whenever $A$ is either~$Q$ or~$S$.

Thus, we see that equality~\eqref{eq_neigh_dihedral} holds true in all cases.

2. Suppose that formula~\eqref{eq_neigh_dihedral} holds true for all pairs~$(U, W)$ such that $|W|\le p$, and formula~\eqref{eq_dihedral} holds true for all pairs~$(I,J)$ such that $|J|=p$. Consider an arbitrary pair~$(I,J)$ such that $|J|=p+1$. Then formula~\eqref{eq_dihedral} for the pair~$(I,J)$ follows immediately from formula~\eqref{eq_dihedral} for the pair~$(I\cup\{l\},J\setminus\{l\})$ and formula~\eqref{eq_neigh_dihedral} for the pair $(I,J\setminus\{l\})$, where  $l$ is an arbitrarily element of~$J$.
\end{proof}

For each $(n-2)$-dimensional face $F=\Delta_{I,J}$, we shall conveniently introduce the notation
\begin{equation*}
\lambda_F=(-1)^{|J\cap X_{k}|}s_{k}\lambda_{k},
\end{equation*}
where $k$ is a unique element of the set~$[n]\setminus(I\cup J)$.
Then $\psi_F(u)=2\arctan(\lambda_Fu)$ whenever $s_{k}s_{k}'=1$, and $\psi_F(u)=2\arctan(\lambda_Fu)+\pi$ whenever $s_{k}s_{k}'=-1$.

\begin{remark}
For flexible cross-polytopes (quadrangles) of the first type in~$\E^2$ or~$\Lambda^2$, the above proof does not work. Indeed, in these cases the quadrangle $L_{\emptyset,\emptyset}(u)=P_u$ would be either Euclidean or hyperbolic rather than spherical, while we have used substantially that it is  spherical, for instance, claiming that the sum of its angles is greater than~$2\pi$ in the case shown in Fig.~\ref{fig_3polyg}(a). In fact, one can show that formulae~\eqref{eq_dihedral}  still hold true for flexible quadrangles of the simplest type in~$\E^2$, but do not hold true for some flexible quadrangles of the simplest type in~$\Lambda^2$. We shall not do this in the present paper, since the case of  quadrangles of the simplest type in~$\E^2$ and~$\Lambda^2$  does not represent a serious self-interest, and will not be used in the sequel.
\end{remark} 
 
\section{Embedded spherical flexible cross-polytopes}\label{section_embed} 

\begin{theorem}\label{theorem_homeo}
Let $P_u\colon K_n\to\bS^n$ be a spherical flexible cross-polytope of the simplest type corresponding to a triple  $(G,\blambda,\bs)$ such that $0<\lambda_1<\cdots<\lambda_n$ and $s_is_i'=-1$ for all~$i$. Then the mapping~$P_0$ is the homeomorphism of~$K_n$ onto the equatorial sphere~$\bS^{n-1}\subset\bS^n$. Hence, the mapping~$P_u$ is an embedding for all sufficiently small~$u$. 
\end{theorem}

\begin{proof}
We shall proof the assertion of the theorem by the induction on the dimension~$n$.

\textsl{Basis of induction:\/} $n=2$. By~\eqref{eq_dihedral}, all angles of the spherical quadrangle~$P_0$ are equal to~$\pi$. Hence this quadrangle is a decomposition of the great circle~$\bS^1\subset\bS^2$ into four arcs. Thus, the assertion of the theorem is true. Notice that in this reasoning it is important that the length of every side of the quadrangle~$P_0$ is strictly less than~$\pi$. This implies that the perimeter of~$P_0$ is less than~$4\pi$, hence $P_0$ cannot ``wind'' on the great circle~$\bS^1$ two or more times.

\textsl{Inductive step.} Suppose that $n\ge 3$. Assume that the assertion of the theorem is true for $m$-dimensional flexible  cross-polytopes of the simplest type for all $m<n$,  and prove the assertion of the theorem for an   $n$-dimensional flexible  cross-polytope of the simplest type~$P_u$. Recall that a continuous mapping $f\colon X\to Y$ of topological spaces is called a \textit{local homeomorphism\/} at a point~$x\in X$, if it maps homeomorphically a neighborhood of~$x$ in~$X$ onto a neighborhood of~$f(x)$ in~$Y$. Let us prove that the mapping $P_0\colon K_n\to\bS^{n-1}$ is a local homeomorphism at all points of~$K_n$. For points in the interiors of $(n-1)$-dimensional simplices, this follows immediately, since all faces of~$P_0$ are non-degenerate. For points in the interiors of $(n-2)$-dimensional simplices, this also follows immediately, since, by~\eqref{eq_dihedral}, the dihedral angles at all $(n-2)$-dimensional faces of~$P_0$ are equal to~$\pi$.
Consider an $(n-k-1)$-dimensional simplex~$\Delta$ of~$K_n$, where $k>1$, and a point~$\bx$ in its relative interior. With some abuse of notation, we denote the face of~$P_u$ that is the image of the simplex~$\Delta$ again by~$\Delta$. The oriented dihedral angles of the link $L(\Delta,P_u)$ at its $(k-2)$-dimensional faces are equal to the oriented dihedral angles of~$P_u$ at the corresponding  $(n-2)$-dimensional faces of it. Therefore, the link $L(\Delta,P_u)$ is itself a flexible spherical cross-polytope of the simplest type, and all its dihedral angles become equal to~$\pi$ for $u=0$. Hence, after a renumbering of its vertices, this cross-polytope corresponds to a triple~$\bigl(\widetilde{G},\widetilde{\blambda},\widetilde{\bs}\bigr)$, where $0<\widetilde{\lambda}_1<\cdots< \widetilde{\lambda}_k$. It follows from~\eqref{eq_dihedral} that $\widetilde{s}_i\widetilde{s}^{\,\prime}_i=-1$, $i=1,\ldots,k$. Consequently, by the inductive assumption, we obtain that the surface~$L(\Delta,P_0)$ is embedded and coincides with the $(k-1)$-dimensional great sphere $\bS^{k-1}_{\bx}$ cut in~$\bS^k_{\bx}$ by the tangent space to the great sphere~$\bS^{n-1}$ containing~$P_0$. Therefore, the mapping $P_0\colon K_n\to\bS^{n-1}$ is a local homeomorphism at~$\bx$.

Thus, the mapping $P_0\colon K_n\to\bS^{n-1}$ is a local homeomorphism at all points of~$K_n$. Since the simplicial complex~$K_n$ is compact, it follows that this mapping is a finite-sheeted non-ramified covering, see, for instance,~\cite[Sect.~V.2]{Mas67}. But $n\ge 3$, hence, the sphere~$\bS^{n-1}$ is simply connected, therefore, this covering is a homeomorphism, cf.~\cite[Sect.~V.6]{Mas67},~\cite[\S 12]{Pra04}. Obviously a polyhedron sufficiently close to an embedded polyhedron is embedded too. Therefore, since the cross-polytope~$P_u$ is embedded for~$u=0$, we obtain that it is embedded for all sufficiently small~$u$.
\end{proof}
 
\begin{remark}
Similarly, it can be proved that the mapping~$P_{\infty}$ is a homeomorphism onto the equatorial sphere~$\bS^{n-1}$ whenever $s_is_i'=1$ for all~$i$.
\end{remark} 
 
\begin{proof}[Proof of Theorem~\ref{theorem_main}.]
Consider a flexible cross-polytope of the simplest type~$P_u$ in $\bS^n$ corresponding to a triple~$(G,\blambda,\bs)$, where $G$ is a positive definite symmetric matrix with units on the diagonal, $0<\lambda_1<\cdots<\lambda_n$, $s_i=-1$ and $s_i'=1$, $i=1,\ldots,n$. This cross-polytope is well defined for almost all pairs~$(G,\blambda)$, see Remark~\ref{rem_all}. By Theorem~\ref{theorem_homeo}, $P_u$ is embedded for sufficiently small~$u$, that is, satisfy property~(i) in Theorem~\ref{theorem_main}. Formulae~\eqref{eq_bd},~\eqref{eq_param_SL} giving a parametrization of the flexion of~$P_u$ imply immediately that $P_u$ also satisfies properties~(ii) and~(iv) in Theorem~\ref{theorem_main} and ``almost satisfies'' property~(iii). Namely, for~$u>0$, ~$P_u$ is contained in the \textit{closed\/} positive hemisphere~$\overline{\bS_+^n}\subset\bS^n$ consisting of all unit vectors that have non-negative scalar products with~$\bm$, and for $u<0$, $P_u$ is contained in the corresponding \textit{closed\/} negative hemisphere~$\overline{\bS_-^n}$.

A flexible cross-polytope that, in addition, satisfy property~(iii) can be obtained by the rotation of the whole sphere~$\bS^n$. In the space~$\R^{n+1}$ containing the sphere~$\bS^n$, we choose a vector~$\mathbf{k}$ orthogonal to~$\bm$ and forming acute angles with the vectors $\bc_1,\ldots,\bc_n$. For instance, we can take $\mathbf{k}=\bn_1+\cdots+\bn_n$. Consider the $(n-1)$-dimensional vector subspace  $U\subset\R^{n+1}$ orthogonal to~$\bm$ and~$\mathbf{k}$.
Denote by $R_{\alpha}$ the rotation around~$U$ by angle~$\alpha$, where the direction of the rotation is chosen so that, for small positive~$\alpha$, the vectors~$R_{\alpha}(\bc_i)$ form obtuse angles with~$\bm$.

For each $u$, we denote by~$\rho(u)$ the smallest of the distances (in the metric of~$\bS^n$) from the vertices $\bb_1(u),\ldots,\bb_n(u)$ of~$P_u$ to the equatorial great sphere~$\bS^{n-1}$ orthogonal to~$\bm$. Then $\rho\colon\overline{\R}\to\R$ is a continuous function such that~$\rho(0)=\rho(\infty)=0$ and $\rho(u)>0$ unless $u= 0,\infty$.

Consider the flexible cross-polytope $$\widetilde{P}_u=R_{\alpha(u)}(P_u),\qquad \alpha(u)=\frac{1}{2}\sign(u)\rho(u).$$ Let us show that it satisfies properties~(i)--(iv) in Theorem~\ref{theorem_main}. Since the cross-polytope~$\widetilde{P}_u$ is obtained from~$P_u$ by a rotation, and besides $\widetilde{P}_0=P_0$ and $\widetilde{P}_{\infty}=P_{\infty}$, the properties~(i), (ii) and~(iv) for~$\widetilde{P}_u$ follow from the same properties for~$P_u$. Let us prove property~(iii). The definition of the rotations~$R_{\alpha}$ implies immediately that the vertices $\tilde\ba_i(u)=R_{\alpha(u)}(\ba_i)$ lie in the open hemisphere~$\bS^n_+$ whenever $u>0$,  and in the open hemisphere~$\bS^n_-$ whenever $u<0$. The vertices~$\bb_i(u)$ of the initial cross-polytope~$P_u$ also lie in~$\bS^n_+$ whenever $u>0$ and in~$\bS^n_-$ whenever $u<0$. Under the rotation by~$\rho(u)/2$ these vertices are shifted by distances that are no more than~$\rho(u)/2$, hence, they cannot leave these hemispheres. Therefore, $\widetilde{P}_u$ is contained in~$\bS^n_+$ whenever $u>0$ and in~$\bS^n_-$ whenever $u<0$.
\end{proof}

\begin{remark}
In some partial cases,  the proof of Theorem~\ref{theorem_homeo} can be obtained without using topological facts and without formulae~\eqref{eq_dihedral} for the oriented dihedral angles. Consider a special case of the cross-polytopes of the simplest type corresponding to the unit matrix $G=E$,  coefficients $0<\lambda_1<\cdots<\lambda_n$, and the signs 
$s_i=-1$ and $s_i'=1$ for all~$i$. In addition, assume that the coefficients~$\lambda_i$ satisfy inequalities 
$\lambda_{i+1}/\lambda_i>2n$, $i=1,\ldots,n-1$. 
For the vectors $\bm,\bn_1,\ldots,\bn_n$ we can take the vectors $\be_0,\ldots,\be_n$ that form the standard orthonormal basis of~$\R^{n+1}$.  From the explicit formulae~\eqref{eq_bd},~\eqref{eq_param_SL} parametrizing the flexion of the cross-polytope under consideration, we can easily deduce the estimates
\begin{equation*}
\dist_{\bS^n}(\ba_i,-\be_i)<\arcsin\frac1{\sqrt{n}}\,,\qquad \dist_{\bS^n}(\bb_i(0),\be_i)<\arcsin\frac1{\sqrt{n}}
\end{equation*} 
Further, it is not hard to show that any cross-polytope satisfying these estimates is embedded. This already implies  Theorem~\ref{theorem_main}. 
\end{remark}

\section{Flat positions}\label{section_flat}

Any flexible cross-polytope~$P_u$ of the simplest type has two  \textit{flat\/} positions, $P_0$ and $P_{\infty}$. The word ``flat'' means  ``contained in a hyperplane $\X^{n-1}\subset\X^n$''. In this section, we shall study the geometric properties of the flat cross-polytopes~$P_{0}$ and~$P_{\infty}$.  We denote by~$P_{(i)}$ the $(n-1)$-dimensional cross-polytope in~$\X^{n-1}$ obtained from~$P_0$ by removing the vertices~$\ba_i$ and~$\bb_i$, i.\,e., the $(n-1)$-dimensional cross-polytope with the vertices $\ba_1,\ldots,\widehat\ba_i,\ldots,\ba_n$, $\bb_1,\ldots,\widehat\bb_i,\ldots,\bb_n$. 

Recall that an  \textit{orisphere\/} in the Lobachevsky space~$\Lambda^m$ is a hypersurface $\Omega$ given in the vector model by the equation $\langle\bx,\bv\rangle=const$  for some isotropic vector $\bv\in\R^{n,1}$. The point on the absolute corresponding to the vector~$\bv$ is called the \textit{centre\/} of the orisphere~$\Omega$.
An \textit{equidistant hypersurface\/} in~$\Lambda^m$ or~$\bS^m$ is a connected component of the set of points lying on the fixed non-zero distance from the given hyperplane~$H$, which is called the \textit{base\/} of this equidistant hypersurface. In~$\bS^m$, equidistant hypersurfaces are small spheres. 

Let $P$ be an $m$-dimensional cross-polytope in~$\X^m$, and let $\Omega$ be a sphere or an orisphere or an equidistant hypersurface in~$\X^m$. We shall say that the cross-polytope~$P$ is  \textit{circumscribed\/} about~$\Omega$ if the hyperplane of every facet of~$P$ tangents~$\Omega$. (We do not require that the tangent point is inside the facet.) A sphere or an orisphere or an equidistant hypersurface about which a cross-polytope~$P$ is circumscribed yields a decomposition of the facets of~$P$ into two classes in the following way. Choose some orientations of the hypersurface~$\Omega$ and of the polyhedral surface~$P$. To the tangent point  of~$\Omega$ and the hyperplane of a facet~$F$ of~$P$ we assign the sign~$+$ if their orientations coincide at the tangent point, and the sign~$-$ otherwise. All facets of~$P$ are decomposed into two classes~$\CF_+$ and $\CF_-$ depending on these signs. If we reverse either the orientation of~$\Omega$ or the orientation of~$P$, then the classes~$\CF_+$ and $\CF_-$ will be interchanged. The obtained decomposition without specifying which of the two classes is positive and which is negative will be called the  \textit{decomposition given by\/~$\Omega$.}  

Now, assume that the hyperplanes of all facets of a cross-polytope~$P$ intersect a hyperplane  $H\subset\X^m$ by the same angle $\alpha\in(0,\pi/2)$.  Since $\alpha<\pi/2$, we see that, at the intersection point of the hyperplane of the facet~$F$ and the hyperplane~$H$, the orthogonal projection provides the correspondence between the orientations of these two hyperplanes. Thus, we again obtain the decomposition of the facets of~$P$ into two classes. This decomposition will be called the \textit{decomposition given by\/~$H$.} Exactly in the same way, one can obtain a  decomposition of the facets of~$P$  into two classes if $\X^m=\Lambda^m$ and the hyperplanes of all facets of~$P$ are parallel to the same hyperplane~$H$. (Recall that two hyperplanes in~$\Lambda^m$ are called \textit{parallel\/} if they have no common points in~$\Lambda^m$ and have exactly one common point on the absolute.)

{\sloppy
\begin{theorem}\label{theorem_opisan}
For each flexible cross-polytope of the simplest type~$P_u$ in\/~$\X^n,$ $n\ge 3,$ we have one of the following  two possibilities:
\begin{enumerate}
\item The cross-polytopes $P_{(1)},\ldots,P_{(n)}$ are circumscribed about concentric  $(n-2)$-dimen\-sio\-nal spheres or orispheres $\Omega_1,\ldots,\Omega_n$  respectively in~$\X^{n-1},$ and besides  then the spheres $\Omega_1,\ldots,\Omega_n$ are not great spheres if\/ $\X^{n-1}=\bS^{n-1}$.
\item There is a hyperplane $H\subset\X^{n-1}$ such that, for each of the cross-polytopes\/ $P_{(1)},\ldots,P_{(n)},$ one of the following three conditions is fulfilled:
\begin{enumerate}
\item  $P_{(k)}$ is circumscribed about an equidistant hypersurface~$\Omega_k$ with base~$H.$ 
\item  \textnormal{(}Only in the case $\X^{n-1}=\Lambda^{n-1}.$\textnormal{)} The hyperplanes of all facets of~$P_{(k)}$ are parallel to~$H.$ 
\item  The hyperplanes of all facets of~$P_{(k)}$ intersect $H$ by the same angle~$\alpha_k\in(0,\pi/2)$.
\end{enumerate}
\end{enumerate}
Besides, for each~$k$, the decomposition of facets of~$P_{(k)}$ into two classes given by~$\Omega_k$ in cases~\textnormal{(i)} and~\textnormal{(ii-a)}, and by~$H$ in cases~\textnormal{(ii-b)} and~\textnormal{(ii-c),}  is as follows: One of the classes consists of all facets~$\Delta_{I,J}$ such that the set $J\setminus X_k$ has even cardinality, and the other class  consists of all facets~$\Delta_{I,J}$ such that $J\setminus X_k$ has odd cardinality. 
\end{theorem}
}

\begin{remark}
If~$\X^n=\bS^n$, then the possibilities~(i) and~(ii) coincide to each other. Indeed, if the cross-polytope~$P_{(k)}$ is circumscribed about an $(n-2)$-dimensional small sphere~$\Omega_k$, then all facets of~$P_{(k)}$ intersect the great sphere~$H$ concentric to~$\Omega_k$ by the same angle~$\alpha_k\in(0,\pi/2)$, and vice versa.  
\end{remark}

\begin{remark}
The assertion of Theorem~\ref{theorem_opisan} remains literally the same if we replace the flat position~$P_0$ of~$P_u$ with the other flat position~$P_{\infty}$ of it. This is not surprising, since it is clear that these two cases are completely similar to each other. However, it is interesting that in both cases the resulting decompositions of facets of the cross-polytopes~$P_{(k)}$ into two classes are governed by the evenness of the cardinalities of the same sets~$J\setminus X_k$. This can be shown in the following way.  Let $(G,\blambda,\bs)$ be the set of data corresponding to the flexible cross-polytope~$P_u$. Consider the flexible cross-polytope~$\widetilde{P}_{\widetilde{u}}$ corresponding to the set of data $(G,\widetilde{\blambda},\tilde\bs)$, where $\widetilde{\lambda}_i=1/\lambda_i$, $\tilde{s}_i=s_i$, and  $\tilde{s}_i'=-s_i'$ for all~$i$. It can be immediately checked that  $P_{u}=\widetilde{P}_{1/u}$ for all~$u$, in particular, $P_{\infty}=\widetilde{P}_0$. Hence, to reduce the assertion of the theorem for~$P_{\infty}$ to the assertion of the theorem for~$\widetilde{P}_0$, it is enough to show that the sets~$X_k$ for the flexible cross-polytopes~$P_u$ and~$\widetilde{P}_{\widetilde{u}}$ are identical to each other. At first sight, this is not correct, since we have reversed all signs $s_i'$. Nevertheless, the coefficients~$\widetilde{\lambda}_i$ are decrease rather than increase. Therefore, all results described above including Theorem~\ref{theorem_opisan}, will become true for~$\widetilde{P}_{\widetilde{u}}$ only after we have renumbered its vertices in the opposite order or, equivalently, have reversed the order on the set~$[n]$. Now, we need only to notice that, simultaneously reversing the order on the set~$[n]$ and  all signs~$s_i'$, we do not change the sets~$X_k$.
\end{remark}

Let~$G=\Delta_{U,W}$ be an arbitrary  $(n-3)$-dimensional face of~$P_{(k)}$, and let~$l$ be a unique element of the set $[n]\setminus(U\cup W\cup\{k\})$.  Let $F$ and~$F'$ be the two facets of~$P_{(k)}$ that contain the face~$G$. Consider the dihedral angle between the $(n-2)$-dimensional simplices~$F$ and~$F'$ in~$\X^{n-1}$. If $l\in X_k$, then we denote by~$\CB_{k,G}$ the \textit{interior\/} bisecting hyperplane of the dihedral angle between~$F$ and~$F'$, and if $l\notin X_k$, then we denote by~$\CB_{k,G}$ the \textit{exterior\/} bisecting hyperplane of the dihedral angle between~$F$ and~$F'$. 

In the case of the Euclidean space~$\E^{n-1}$, we consider its projectivization~$\RP^{n-1}$, and denote by~$\hCB_{k,G}$ the projectivization of the hyperplane~$\CB_{k,G}$. In the case of the Lobachevsky space~$\Lambda^{n-1}$, we consider its Beltrami--Klein model in which it is identified with a disk in the projective space~$\RP^{n-1}$, and denote by~$\hCB_{k,G}$ the projective hyperplane in~$\RP^{n-1}$ containing the hyperplane $\CB_{k,G}\subset\Lambda^{n-1}$. In the spherical case, we denote by~$\hCB_{k,G}$ the image of the great sphere~$\CB_{k,G}$ under the natural two-sheeted projection  $\bS^{n-1}\to\RP^{n-1}$. 

\begin{lem}\label{lem_intersect}
The projective hyperplanes $\hCB_{k,G}$ corresponding to all pairs~$(k,G)$ such that $G$ is an $(n-3)$-dimensional face of~$P_{(k)}$ intersect exactly in one point.  
\end{lem}

\begin{lem}\label{lem_nondeg}
Any dihedral angle of any of the cross-polytopes~$P_{(k)}$ is equal neither to~$0$ nor to~$\pi$.
\end{lem}

We shall show that Theorem~\ref{theorem_opisan}  follows from Lemmas~\ref{lem_intersect} and~\ref{lem_nondeg}, and then we shall prove these lemmas.

\begin{proof}[Proof of Theorem~\ref{theorem_opisan}]
Let $O$ be the intersection points of the projective hyperplanes~$\hCB_{k,G}$, which exists and is unique by Lemma~\ref{lem_intersect}. Consider the cases:

1. Suppose that~$O$ lies in~$\X^{n-1}=\E^{n-1}$ or~$\Lambda^{n-1}$ or corresponds to a pair of antipodal points of~$\X^{n-1}=\bS^{n-1}$. In the latter case, we denote by~$O$ one of the two antipodal points of~$\bS^{n-1}$ projecting to~$O\in\RP^{n-1}$. Then, for each pair~$(k,G)$, the bisecting hyperplane~$\CB_{k,G}$ of the dihedral angle between the two  $(n-2)$-dimensional faces~$F$ and~$F'$ of~$P_{(k)}$ containing~$G$ passes through~$O$. Hence, each cross-polytope~$P_{(k)}$ is circumscribed about a sphere~$\Omega_k$ with centre~$O$.

2. Suppose that $\X^{n-1}=\Lambda^{n-1}$ and $O$ is a point on the absolute. Similarly, we obtain that each cross-polytope~$P_{(k)}$ is circumscribed about an orisphere~$\Omega_k$ with centre~$O$.

3. Suppose that $\X^{n-1}=\E^{n-1}$ and $O$ is a point at infinity. Consider an arbitrary hyperplane $H\subset\E^{n-1}$ orthogonal to lines passing through~$O$. Then, for each pair~$(k,G)$, the bisecting hyperplane~$\CB_{k,G}$ of the dihedral angle between the facets~$F$ and~$F'$ of~$P_{(k)}$ containing~$G$ is  perpendicular to~$H$.  Therefore all facets of each cross-polytope~$P_{(k)}$ form the same angle with~$H$.

4. Suppose that $\X^{n-1}=\Lambda^{n-1}$ and $O\in\RP^{n-1}$ is a point outside the absolute.  We denote by~$H\subset\Lambda^{n-1}$ the hyperplane that is the polar of~$O$ with respect to the absolute. Then, for each pair~$(k,G)$, the bisecting hyperplane~$\CB_{k,G}$ between the facets~$F$ and~$F'$ of~$P_{(k)}$ containing~$G$ intersects the hyperplane~$H$, and is perpendicular to it. Hence either the hyperplanes of both facets~$F$ and~$F'$ are divergent with~$H$ and are on the same distance from it or the hyperplanes of both facets~$F$ and~$F'$ are parallel to~$H$ or the hyperplanes of both facets~$F$ and~$F'$ intersect~$H$ by the same angle. Hence, for each of the cross-polytopes~$P_{(k)}$, one of the assertions~(ii-a)--(ii-c) in Theorem~\ref{theorem_opisan} holds.

It follows from Lemma~\ref{lem_nondeg} that, in case~1, none of the spheres~$\Omega_k$ is a great sphere in~$\bS^{n-1}$, and in the cases~$3$ and~$4$, the facets of~$P_{(k)}$ are neither contained in~$H$ nor perpendicular to~$H$.

In each of the cases considered, the facets~$F$ and~$F'$ of~$P_{(k)}$ belong to the same class in the decomposition corresponding either to the hypersurface~$\Omega_k$ or to the hyperplane~$H$ if and only if $\CB_{k,G}$ is the interior bisecting hyperplane of the dihedral angle between~$F$ and~$F'$, i.\,e., if and only if $l\in X_k$, where   $l$ is a unique element of the set $[n]\setminus(U\cup W\cup\{k\})$,  $G=\Delta_{U,W}$. This easily implies that one of the two classes consists of all facets~$\Delta_{I,J}$ of~$P_{(k)}$ with  even cardinality $|J\setminus X_k|$, and the other class consists of all facets~$\Delta_{I,J}$ of~$P_{(k)}$ with odd cardinality $|J\setminus X_k|$. 
\end{proof}

Let $F$ and~$F'$ be two  $(n-2)$-dimensional simplices in~$\X^{n-1}$ with a common $(n-3)$-dimensional face~$G$, and let $H\subset\X^{n-1}$ be an arbitrary hyperplane passing through~$G$. Choose a co-orientation  of~$G$ in~$\X^{n-1}$, that is, the direction of the positive circuit around it, and choose one of the two half-planes $H_+\subset H$ bounded by the plane of~$G$. Define the oriented angle $\angle(F,H_+)\in\R/(2\pi\Z)$ to be the angle of the rotation of~$F$ in the positive direction around~$G$ to the half-plane~$H_+$, and put $\angle(H_+,F_i)=-\angle(F,H_+)$,   and similarly for~$F'$. Then the ratio
$$
r(F,H,F')=\frac{\sin\angle(F,H_+)}{\sin\angle(H_+,F')}
$$
is determined solely by the simplices~$F$ and~$F'$ and the hyperplane~$H$, and is independent of the choice of the co-orientation of~$G$ and of the choice of the half-plane~$H_+$. Besides, 
$$r(F,H,F')=r(F',H,F)^{-1}.$$

Now, we take for~$G$ an arbitrary  $(n-3)$-dimensional face~$\Delta_{U,W}$ of the cross-polytope~$P_0$. Let $\{k,l\}=[n]\setminus(U\cup W)$. Then~$G$ is a face of the cross-polytopes~$P_{(k)}$ and~$P_{(l)}$. Let $F_1$ and~$F_1'$ be the two facets of~$P_{(k)}$ containing~$G$, and let $F_2$ and~$F_2'$ be the two facets of~$P_{(l)}$ containing~$G$. (Obviously, $F_1$ and $F_1'$ are exactly the faces~$\Delta_{U\cup\{l\},W}$ and~$\Delta_{U,W\cup\{l\}}$, but we do not want to specify which of them is~$F_1$, and which of them is~$F_1'$ to avoid considering several cases in the sequel; similarly for~$F_2$ and~$F_2'$.)

\begin{lem}\label{lem_CB}
The hyperplanes~$\CB_{k,G}$ and~$\CB_{l,G}$ coincide to each other, and
\begin{equation*}
r(F_1,\CB_G,F_2)=\frac{\lambda_{F_2}}{\lambda_{F_1}}\,,
\end{equation*}
where $\CB_G=\CB_{k,G}=\CB_{l,G}$.
\end{lem}

\begin{proof}
Suppose that $k<l$. The link $L_G(u)=L(G,P_u)$ is the flexible spherical quadrangle with the vertices $A$, $B$, $C$, and~$D$ cut by the tangents cones to the faces~$F_1$, $F_2$, $F_1'$, and~$F_2'$ respectively in the sphere~$\bS^2_{\bx}$ for an interior point~$\bx$ of~$G$. By Lemma~\ref{lem_dihedral}, the tangents of the halves of the oriented dihedral angles~$\psi_A(u)=\psi_{F_1}(u)$ and~$\psi_B(u)=\psi_{F_2}(u)$ are either directly or inversely proportional to each other. Hence the flexible quadrangle~$L_G(u)$ is of one of the three types shown in Fig.~\ref{fig_3polyg}. Denote by~$\alpha$ and~$\beta$ the length of the sides~$AB$ and~$AD$ respectively. Then the lengths of the sides~$CD$ and~$BC$ are equal to~$\alpha$ and~$\beta$  respectively for the quadrangles in Fig.~\ref{fig_3polyg}(a),\,(b), and are equal to $\pi-\alpha$ and~$\pi-\beta$  respectively for the quadrangle in Fig.~\ref{fig_3polyg}(c). Besides,~$\alpha\ne\beta$ and $\alpha+\beta\ne\pi$. Indeed, it is easy to check that, if one of these two equalities held true, then one of the  angles of~$L_G(u)$, i.\,e., one of the dihedral angles of~$P_u$ would be either identically~$0$ or identically~$\pi$ during the flexion, which is impossible, since $\lambda_k\ne 0$ and~$\lambda_l\ne 0$.
Bricard~\cite[\S II]{Bri97}  showed that, for each of the quadrangles in Fig.~\ref{fig_3polyg}(a),\,(b), we have one of the  two equalities
\begin{align}
\tan\frac{\psi_A(x)}{2}\,\tan\frac{\psi_B(x)}{2}&=
\frac{\cos\frac{\alpha-\beta}{2}}{\cos\frac{\alpha+\beta}{2}}\ ,\label{eq_tg_prod1}&
\tan\frac{\psi_A(x)}{2}\,\tan\frac{\psi_B(x)}{2}&=
\frac{\sin\frac{\beta-\alpha}{2}}{\sin\frac{\alpha+\beta}{2}}\ ,
\end{align}
and for a quadrangle in Fig.~\ref{fig_3polyg}(c) we have one of the  two equalities
\begin{align*}
\frac{\tan\frac{\psi_A(x)}{2}}{\tan\frac{\psi_B(x)}{2}}&=
\frac{\sin\frac{\alpha-\beta}{2}}{\sin\frac{\alpha+\beta}{2}}\ ,&
\frac{\tan\frac{\psi_A(x)}{2}}{\tan\frac{\psi_B(x)}{2}}&=
-\frac{\cos\frac{\alpha-\beta}{2}}{\cos\frac{\alpha+\beta}{2}}\ .
\end{align*}

Consider the flat position $L_G(0)$ of~$L_G(u)$ lying in the great circle  $\bS^1_{\bx}\subset\bS^2_{\bx}$ cut by the tangent space~$T_{\bx}\X^{n-1}$. We shall study in detail the case of the quadrangle~$L_G(u)$ shown in Fig.~\ref{fig_3polyg}(a), the two other cases are completely similar. In the case in Fig.~\ref{fig_3polyg}(a), we have $s_ks_k'=-s_ls_l'$. Consider two subcases:

1. Suppose that $s_ks_k'=1$ and $s_ls_l'=-1$. Then $k\in X_l$ and~$l\in X_k$. Hence, $\CB_{k,G}$ and $\CB_{l,G}$ are interior bisecting hyperplanes of the dihedral angles between~$F_1$ and~$F_1'$ and between~$F_2$ and~$F_2'$ respectively. Therefore, the tangent space $T_{\bx}\CB_{k,G}$ intersects the circle~$\bS^1_{\bx}$ in the midpoints of the two arcs with endpoitns~$A$ and~$C$,  and the tangent space $T_{\bx}\CB_{l,G}$ intersects~$\bS^1_{\bx}$ in the midpoints of the two arcs with endpoitns~$A$ and~$C$. The flat quadrangle $L_{G}(0)$ in the circle~$\bS^1_{\bx}$ has the form shown in Fig.~\ref{fig_4circ}(a) if  $\alpha>\beta$, and the form shown in Fig.~\ref{fig_4circ}(b) if $\alpha<\beta$. In both cases the midpoints~$E$ and~$E'$ of the arcs with endpoints~$A$ and~$C$ coincide with the midpoints of the arcs with endpoints~$B$ and~$D$, see Fig.~\ref{fig_4circ}(a),\,(b). Hence, $\CB_{k,G}=\CB_{l,G}$. The ration $r(F_1,\CB_G,F_2)$ is equal to the ration of the sines of the oriented lengths of the arcs~$AE$ and~$EB$. Therefore,
$$
r(F_1,\CB_G,F_2)=\frac{\sin\frac{\alpha+\beta}{2}}{\sin\frac{\alpha-\beta}{2}}\ .
$$
Since $s_ks_k'=1$ and $s_ls_l'=-1$, we have $\tan\frac{\psi_A(x)}{2}=\lambda_{F_1}x$ and $\tan\frac{\psi_B(x)}{2}=-\lambda^{-1}_{F_2}x^{-1}$. Hence it follows from~\eqref{eq_tg_prod1} that the ration $\frac{\lambda_{F_2}}{\lambda_{F_1}}$ is equal either to $-\frac{\cos\frac{\alpha+\beta}{2}}{\cos\frac{\alpha-\beta}{2}}$ or to $\frac{\sin\frac{\alpha+\beta}{2}}{\sin\frac{\alpha-\beta}{2}}\,$. Since $\alpha,\beta\in(0,\pi)$, we easily obtain that 
$$
\left|
\frac{\cos\frac{\alpha+\beta}{2}}{\cos\frac{\alpha-\beta}{2}}
\right|<1,\qquad
\left|
\frac{\sin\frac{\alpha+\beta}{2}}{\sin\frac{\alpha-\beta}{2}}
\right|>1.
$$
But $|\lambda_{F_1}|<|\lambda_{F_2}|$, since $k<l$. Thus,
$$
\frac{\lambda_{F_2}}{\lambda_{F_1}}=\frac{\sin\frac{\alpha+\beta}{2}}{\sin\frac{\alpha-\beta}{2}}=r(F_1,\CB_G,F_2).
$$
\begin{figure}
\begin{center}
\includegraphics[scale=.5]{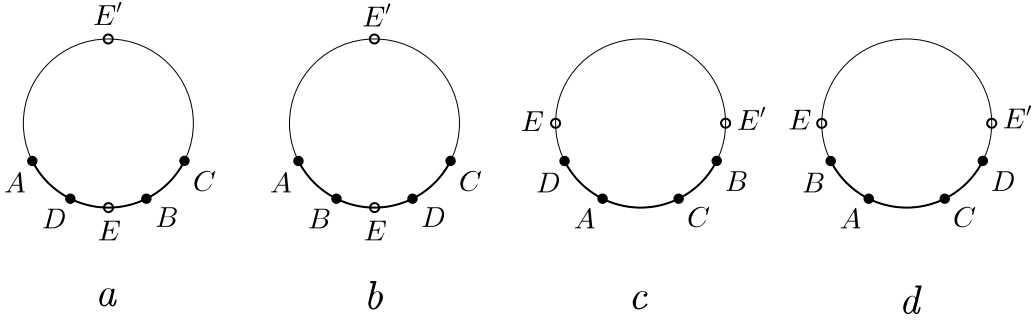}
\end{center}
\caption{The flat position of the link of~$G$\,: (a) $s_ks_k'=1$, $\alpha>\beta$, (b)~$s_ks_k'=1$, $\alpha<\beta$, (c) $s_ks_k'=-1$, $\alpha>\beta$, (d) $s_ks_k'=-1$, $\alpha<\beta$.} \label{fig_4circ}
\end{figure}

2. Suppose that $s_ks_k'=-1$ and $s_ls_l'=1$. Then $k\notin X_l$ and~$l\notin X_k$. Hence, $\CB_{k,G}$ and $\CB_{l,G}$ are the exterior bisecting hyperplanes of the dihedral angles between~$F_1$ and~$F_1'$ and between~$F_2$ and~$F_2'$ respectively. Therefore, the tangent space $T_{\bx}\CB_{k,G}$ intersects~$\bS^1_{\bx}$ in the endpoints of the diameter parallel to the chord~$AC$, and the tangent space $T_{\bx}\CB_{l,G}$ intersects~$\bS^1_{\bx}$ in the endpoints of the diameter parallel to the chord~$BD$. The flat quadrangle $L_{G}(0)$  has the form shown in Fig.~\ref{fig_4circ}(c) if $\alpha>\beta$, and has the form shown in Fig.~\ref{fig_4circ}(d) if $\alpha<\beta$. In both cases, the chords $AC$ and~$BD$ are parallel to each other, hence, the same diameter $EE'$ is parallel to both of them. Therefore, $\CB_{k,G}=\CB_{l,G}$ and
$$
r(F_1,\CB_G,F_2)=-\frac{\sin\frac{\pi+\alpha-\beta}{2}}{\sin\frac{\pi-\alpha-\beta}{2}}=-\frac
{\cos\frac{\alpha-\beta}{2}}{\cos\frac{\alpha+\beta}{2}}=\frac{\lambda_{F_2}}{\lambda_{F_1}}\ .
$$
The latter equality follows from the formulae $\tan\frac{\psi_A(x)}{2}=-\lambda_{F_1}^{-1}x^{-1}$ and $\tan\frac{\psi_B(x)}{2}=\lambda_{F_2}x$, formulae~\eqref{eq_tg_prod1}, and the inequality $|\lambda_{F_1}|<|\lambda_{F_2}|$.
\end{proof}

\begin{proof}[Proof of Lemma~\ref{lem_nondeg}]
As it was mentioned in the proof of the previous lemma, the link $L_G(u)$ of every $(n-3)$-dimensional face~$G$ of~$P_u$ has the form of one of the quadrangles shown in Fig.~\ref{fig_3polyg}, and besides, the lengths~$\alpha$ and~$\beta$ of its sides~$AB$ and~$AD$ are not equal to each other, and their sum is not equal to~$\pi$. Hence, in the flat position~$L_G(0)$ of this quadrangle any pair of its vertices neither coincide to each other nor are antipodal to each other. Therefore, the dihedral angles of the cross-polytopes~$P_{(k)}$ and~$P_{(l)}$ at~$G$ are neither zero nor straight. 
\end{proof}

\begin{lem}\label{lem_3CB}
Let $F_1$, $F_2$, and~$F_3$ be three pairwise distinct $(n-2)$-dimensional faces of the cross-polytope~$P_0$ that are contained in the same $(n-1)$-dimensional face of~$P_0$. Then
\begin{equation}\label{eq_CB}
\hCB_{F_1\cap F_2}\cap \hCB_{F_2\cap F_3}=
\hCB_{F_2\cap F_3}\cap \hCB_{F_3\cap F_1}=
\hCB_{F_3\cap F_1}\cap \hCB_{F_1\cap F_2}.
\end{equation}
\end{lem}
\begin{proof}
We put $G_{ij}=F_i\cap F_j$ and $Q=F_1\cap F_2\cap F_3$. Then $\dim G_{ij}=n-3$ and $\dim Q=n-4$. Suppose that $n\ge 4$. Let $\bx$ be a point in the relative interior of~$Q$. In the space $T_{\bx}\X^{n-1}$ we consider the sphere~$\bS^2_{\bx}$ consisting of all unit vectors orthogonal to~$Q$. The tangent cones to the faces~$F_1$, $F_2$, and~$F_3$ intersect the sphere~$\bS^2_{\bx}$ by arcs of great circles, which we denote by~$f_1$, $f_2$, and~$f_3$ respectively, and the tangent spaces to the hyperplanes $\CB_{G_{12}}$, $\CB_{G_{23}}$, and~$\CB_{G_{31}}$ intersect~$\bS^2_{\bx}$ by great circles, which we denote by~$b_{12}$, $b_{23}$, and~$b_{31}$ respectively. Each great circle~$b_{ij}$ passes through the common endpoint of the arcs~$f_i$ and~$f_j$. By Lemma~\ref{lem_CB}, we have
$$
r(f_i,b_{ij},f_j)=r(F_i,\CB_{G_{ij}},F_j)=\frac{\lambda_{F_j}}{\lambda_{F_i}}.
$$
Hence,
$$
r(f_1,b_{12},f_2)\,r(f_2,b_{23},f_3)\,r(f_3,b_{31},f_1)=1.
$$
By spherical Ceva's theorem in the trigonometric form, the great circles~$b_{12}$, $b_{23}$, and~$b_{31}$ intersect in a pair of antipodal points of~$\bS^2_{\bx}$. Besides, since every of the ratios $r(f_i,b_{ij},f_j)$ is neither zero nor infinity, no pair of the great circles~$b_{12}$, $b_{23}$, and~$b_{31}$ coincide to each other. This immediately implies~\eqref{eq_CB}. If $n=3$, then equality~\eqref{eq_CB} follows in the same way from the trigonometric form of Ceva's theorem in~$\X^2$ applied to the lines $\CB_{G_{12}}$, $\CB_{G_{23}}$, and~$\CB_{G_{31}}$.
\end{proof}

\begin{proof}[Proof of Lemma~\ref{lem_intersect}]
Let $F$ be an $(n-2)$-dimensional face of the cross-polytope~$P_0$. We denote by~$O_F$ the intersection of the $n-1$ projective hyperplanes $\hCB_G\subset\RP^{n-1}$, where $G$ runs over all $(n-3)$-dimensional faces of~$F$. Then~$O_F$ is non-empty.  Let us show that~$O_F$ is a point. If this were not correct, then the intersection of the projective plane~$O_F$ with the projective hyperplane of the face~$F$ would also be non-empty. This is impossible, since the intersection of the projective hyperplane of the face~$F$ with~$\hCB_G$ is exactly the projective plane of the face~$G$, and the intersection of the projective planes of all $(n-3)$-dimensional faces $G\subset F$ is empty. (No hyperplane~$\hCB_G$ can coincide with the hyperplane of~$F$, since the coefficient~$\lambda_F$ is neither zero nor infinity.) To prove the lemma, it remains to show that all points~$O_F$ coincide. To show this, it is sufficient to show that $O_{F_1}=O_{F_2}$ whenever $F_1$ and~$F_2$ are two   $(n-2)$-dimensional faces of~$P_0$ that are contained in an $(n-1)$-dimensional face~$\Delta$ of~$P_0$. Put $G_0=F_1\cap F_2$. For each $(n-2)$-dimensional face~$F\subset \Delta$ such that $F\ne F_1,F_2$, Lemma~\ref{lem_3CB} implies that  $\hCB_{F_1\cap F}\cap\hCB_{G_0}=\hCB_{F_2\cap F}\cap\hCB_{G_0}$. As $F$ runs over all $(n-2)$-dimensional faces~$F\subset\Delta$ such that $F\ne F_1,F_2$, the intersection $F_1\cap F$ runs over all $(n-3)$-dimensional faces of~$F_1$ different from~$G_0$, and the intersection $F_2\cap F$ runs over all $(n-3)$-dimensional faces of~$F_2$ different from~$G_0$. Therefore, the intersection of all hyperplanes~$\hCB_G$, where $G$ runs over all $(n-3)$-dimensional faces of~$F_1$, coincides with the intersection of all hyperplanes~$\hCB_G$, where $G$ runs over all $(n-3)$-dimensional faces of~$F_2$, i.\,e.,~$O_{F_1}=O_{F_2}$.
\end{proof}

\section{Volumes}\label{section_volume}

Recall that the $n$-dimensional volume of the unit $n$-dimensional sphere~$\bS^n$ is equal to
$$\sigma_n=\frac{2\pi^{\frac{n+1}2}}{\Gamma(\frac{n+1}2)}\,.$$
We shall denote the $k$-dimensional volume of the   $k$-dimensional simplex~$\Delta$ by~$V_k(\Delta)$. For each face~$\Delta_{I,J}$ of the cross-polytope~$P_u$, we denote its $(|I|+|J|-1)$-dimensional volume by~$V_{I,J}$.

\subsection{Relations on the volumes of  $(n-1)$-dimensional faces} 

For the flat positions~$P_0$ and~$P_{\infty}$ of a flexible cross-polytope of the simplest type~$P_u$ in~$\bS^n$, we can consider the degrees~$\deg P_0$ and~$\deg P_{\infty}$ of the mappings
$P_0,P_{\infty}\colon K_n\to\bS^{n-1}$ respectively. Recall that, by definition, the \textit{degree\/} is the algebraic number of pre-images of an arbitrary regular value of the mapping considered, where to each point in the pre-image is assigned the sign plus if the differential of the mapping at this point preserves the orientation, and the sign minus otherwise.  All necessary facts on the degrees of mappings can be found, for instance, in~\cite[Sect.~8.7,~18.1]{Pra04}. To fix the signs of the degrees~$\deg P_0$ and~$\deg P_{\infty}$, we need to choose the orientation of the great sphere~$\bS^{n-1}$. We agree to choose this orientation in such a way that the (coinciding to each other) restrictions of the mappings~$P_0$ and~$P_{\infty}$ to the simplex $[\ba_1\ldots\ba_n]$ preserve the orientation.

\begin{lem}\label{lem_deg}
The degree\/ $\deg P_0$ is equal to\/~$1$ if\/ $s_is_i'=-1$ for all\/~$i$, and is equal to\/~$0$ if at least one of the numbers\/~$s_is_i'$ is equal to\/~$1$. The degree\/ $\deg P_{\infty}$ is equal to\/~$1$ if\/ $s_is_i'=1$ for all\/~$i$, and is equal to\/~$0$ if at least one of the numbers\/~$s_is_i'$ is equal to\/~$-1$.
\end{lem}

\begin{proof}
We shall prove the first assertion of the lemma. The proof of the second assertion is completely similar. It follows easily from Theorem~\ref{theorem_homeo} that $\deg P_0=1$ if $s_is_i'=-1$ for all~$i$. 

Assume that $s_ks_k'=1$ for some~$k$. Consider the flexible cross-polytope $\widetilde{P}_u$ obtained from~$P_u$ by replacing all its vertices~$\bb_i(u)$ such that $s_is_i'=1$ with the antipodal points. Then the signs corresponding to the flexible cross-polytope $\widetilde{P}_u$ are $\tilde{s}_i=s_i$ and $\tilde{s}_i'=-s_i$ for all~$i$. Hence, by Theorem~\ref{theorem_homeo}, the cross-polytope $\widetilde{P}_0$ is embedded. We shall denote the vertices~$\ba_i(0)$ and $\bb_i(0)$ of the flat cross-polytope~$P_0$ simply by~$\ba_i$ and~$\bb_i$ respectively. The vertices of the cross-polytope~$\widetilde{P}_0$ are the points~$\ba_i$ and~$\tilde{\bb}_i=-s_is_i'\bb_i$. Let us show that the point  $\tilde\bb_k=-\bb_k$ does not lie in the image of the mapping~$P_0$, that is, in the union of faces of the cross-polytope~$P_0$. Assume the converse. Then $\tilde\bb_k\in\Delta_{A,B}$ for some facet~$\Delta_{A,B}$ of~$P_0$, $A\sqcup B=[n]$. Therefore, 
$$
\tilde\bb_k=\sum_{i\in A}\mu_i\ba_i+\sum_{i\in B}\mu_i\bb_i
$$
for some non-negative coefficients~$\mu_i$. Denote by~$B_+$ and~$B_-$ the subsets of~$B$ consisting of all~$i$ such that $s_is_i'=1$ and~$s_is_i'=-1$ respecttively. Then
$$
\tilde\bb_k+\sum_{i\in B_+}\mu_i\tilde\bb_i=\sum_{i\in A}\mu_i\ba_i+\sum_{i\in B_-}\mu_i\tilde\bb_i.
$$
Since the vectors $\bb_1,\ldots,\bb_n$ are linearly independent, the vector~$\bv$ standing in both sides of this equality is non-zero. Then the point~$\bv/|\bv|$ lies in the faces~$\widetilde{\Delta}_{\emptyset,B_+\cup\{k\}}$ and $\widetilde{\Delta}_{A,B_-}$ of~$\widetilde{P}_0$, which is impossible, since these faces do not have common vertices and the cross-polytope~$\widetilde{P}_0$ is embedded. (The number $k$ can either belong or not belong to~$B_+$, but, for sure, does not belong to~$B_-$). The obtained contradiction shows that the image of the mapping $P_0\colon K_n\to\bS^{n-1}$ does not coincide with the whole sphere~$\bS^{n-1}$, which implies immediately that $\deg P_0=0$.
\end{proof}

We denote by~$Y_+ $ (respectively, by $Y_-$) the subset of~$[n]$ consisting of all~$i$ such that $s_is_i'=1$ (respectively, $s_is_i'=-1$); then $Y_+\sqcup Y_-=[n]$.

\begin{theorem}\label{theorem_n-1}
Let\/ $P_u$ be a flexible cross-polytope of the simplest type in\/~$\X^n,$ and let\/ $Y$ be one of the two subsets\/~$Y_+$ and\/~$Y_-$ corresponding to it. Then 
\begin{equation}\label{eq_n-1}
\sum_{A\sqcup B=[n]}(-1)^{|B\cap Y|}V_{A,B}=\left\{
\begin{aligned}
&0&&\text{if\/ $\X^n=\E^n$ or\/~$\Lambda^n,$}\\
&0&&\text{if\/ $\X^n=\bS^n$ and\/ $Y\ne\emptyset,$}\\
&\sigma_{n-1}&&\text{if\/ $\X^n=\bS^n$ and\/ $Y=\emptyset.$}
\end{aligned}
\right.
\end{equation}
\end{theorem}

\begin{proof}
We shall prove relation~\eqref{eq_n-1} for $Y=Y_+$ by studying the flat position~$P_0$ of the cross-polytope. Relation~\eqref{eq_n-1} for $Y=Y_-$ is obtained in the same way  by studying the flat position~$P_{\infty}$. If $Y_+=\emptyset$, then it follows from Theorem~\ref{theorem_homeo} that the sum of the volumes of the facets of~$P_0$ is equal to the volume~$\sigma_{n-1}$ of~$\bS^{n-1}$, which yields~\eqref{eq_n-1}. 

Assume that $\X^n=\E^n$ or~$\Lambda^n$ or $\X^n=\bS^n$ and $Y_+\ne\emptyset$. Then the degree of the mapping $P_0\colon K_n\to\X^{n-1}$ is equal to zero. For $\X^n=\bS^n$, this follows from Lemma~\ref{lem_deg}. For $\X^n=\E^n$ and~$\Lambda^n$, this is also true, since the image $P_0(K_n)$ is compact, hence, the mapping~$P_0$ is not surjective. Therefore, a generic point in~$\X^{n-1}$ is covered by the facets of~$P_0$ embedded into~$\X^{n-1}$ with the embedding preserving the orientation as many times as by the facets of~$P_0$ embedded into~$\X^{n-1}$ with the embedding  reversing the orientation. Hence the sum of the volumes of the facets  of~$P_0$ embedded into~$\X^{n-1}$ with the embedding preserving the orientation is equal to the sum of the volumes of the facets  of~$P_0$ embedded into~$\X^{n-1}$ with the embedding reversing the orientation. The embedding of a facet~$\Delta_{A,B}$ into~$\X^{n-1}$ preserves the orientation if and only if it is possible to travel from~$\Delta_{[n],\emptyset}$ to~$\Delta_{A,B}$ in~$K_n$ crossing finitely many times the $(n-2)$-dimensional faces of~$P_0$, and besides, crossing even number of times the $(n-2)$-dimensional faces of~$P_0$ the dihedral angles at which are equal to~$0$. It follows from formula~\eqref{eq_dihedral} that this is the case if and only if the number $|B\cap Y_+|$ is even, which yields formula~$\eqref{eq_n-1}$. 
\end{proof}

\begin{cor}
If\/ $\X^n=\E^n$ or\/~$\Lambda^n,$ than not all of the numbers~$s_is_i'$ are the same. 
\end{cor}

\subsection{A property of circumscribed cross-polytopes}
Let $P$ be an $m$-dimensional cross-polytope in~$\X^m$, $m\ge 2$, such that  $P$ is circumscribed about a hypersurface~$\Omega$ which is a sphere or an orisphere or an equidistant hypersurface, or the hyperplanes of all facets of~$P$ form the same angle $\alpha\in(0,\pi/2)$ with some hyperplane~$H$, or $\X^m=\Lambda^m$ and the hyperplanes of all facets of~$P$ are parallel to the same hyperplane~$H$. Let $\CF=\CF_+\sqcup\CF_-$ be the decomposition of the set of facets of~$P$ into two classes given by the hypersurface~$\Omega$ or by the hyperplane~$H$. Consider the chess colouring of facets of~$P$, and denote by~$\CF_b$ and~$\CF_w$ the sets of black and white facets respectively.

It is well known that for a quadrangle circumscribed about a circle the sums of the lengths of its opposite sides are equal to each other. The following lemma is a direct generalization of this fact.

 \begin{lem}\label{lem_rel_basic}
If\/ $\X^m=\E^m$ or~$\Lambda^m$, then
\begin{equation}\label{eq_rel_basic}
\sum_{F\in\CF_+\cap\CF_b}V_{m-1}(F)-\sum_{F\in\CF_-\cap\CF_b}V_{m-1}(F)-\sum_{F\in\CF_+\cap\CF_w}V_{m-1}(F)+\sum_{F\in\CF_-\cap\CF_w}V_{m-1}(F)=0.
\end{equation}
\end{lem}
 
\begin{proof}
The proof is the same as for the two-dimensional case. For each facet~$F$, we denote by~$A_{F}$ either the tangent point of the hyperplane of~$F$ and the hypersurface~$\Omega$ or the intersection point of the hyperplane of~$F$ and the hyperplane~$H$.  (If $\X^m=\Lambda^m$ and the hyperplane of~$F$ is parallel to~$H$, then the point~$A_F$ lies on the absolute.) Then the volume of the simplex ~$F$ can be decomposed into the algebraic sum of the volumes of the simplices~$[A_FG]$ spanned by the point~$A_F$ and facets~$G$ of the simplex~$F$:
\begin{equation}\label{eq_razl}
V_{m-1}(F)=\sum_{G\subset F,\,\dim G=m-2} \pm V_{m-1}([A_FG]).
\end{equation}
It is not hard to check that signs in formulae~\eqref{eq_rel_basic} and~\eqref{eq_razl} are consistent so that the volumes of simplices $[A_{F_1}G]$ and $[A_{F_2}G]$ enter the left-hand side of~\eqref{eq_rel_basic} with the opposite signs for any two facets~$F_1$ and~$F_2$ with a common $(m-2)$-dimensional face~$G$. But the simplices $[A_{F_1}G]$ and $[A_{F_2}G]$ are congruent to each other. Hence their volumes are equal to each other, which implies equality~\eqref{eq_rel_basic}.
\end{proof} 

If~$\X^m=\bS^m$, then formula~\eqref{eq_rel_basic} is in general not correct. The matter is that formula~\eqref{eq_razl} is not correct. Namely, the left-hand side and the right-hand side of~\eqref{eq_razl} can differ by the volume~$\sigma_{m-1}$  of~$\bS^{m-1}$. Nevertheless, there is an important special case in which formula~\eqref{eq_rel_basic} is correct. 
\begin{lem}\label{lem_rel_basic_S_strong}
Suppose that\/ $\X^m=\bS^m$. Let\/ $\overline{\bS^m_+}\subset\bS^m$ be a closed
 hemisphere bounded by the\/ $(m-1)$-dimensional great sphere concentric with the small sphere\/~$\Omega$ about which the cross-polytope\/~$P$ is circumscribed. Assume that\/ $P$ is contained in\/~$\overline{\bS^m_+}$. Then formula~\eqref{eq_rel_basic} holds true for~$P$.
\end{lem}
\begin{proof}
It is sufficient to note that the point~$A_F$ is the projection of the centre of the sphere~$\Omega$ onto the $(m-1)$-dimensional great sphere~$\CH_F$ containing~$F$. Hence, the point $A_F$  lies in the open $(m-1)$-dimensional hemisphere $\CH_{F,+}=\CH_F\cap\bS^m_+$ whose closure contains~$F$. Therefore, formula~\eqref{eq_razl} holds true in this case, and the proof of Lemma~\ref{lem_rel_basic} can be repeated literally. 
\end{proof}

\subsection{Relations on the volumes of $(n-2)$-dimensional faces} 

\begin{theorem}\label{theorem_n-2}
Suppose that\/ $n\ge 3$. Then, for a flexible cross-polytope of the simplest type\/~$P_u$  either in\/~$\E^n$ or in\/~$\Lambda^n,$ we have
\begin{equation}\label{eq_cod2_rel_EL}
\sum_{I\sqcup J=[n]\setminus\{k\}}(-1)^{|J\cap X_k|}V_{I,J}=0, \qquad k=1,\ldots,n,
\end{equation} 
and for a flexible cross-polytope of the simplest type\/~$P_u$  in\/~$\bS^n,$ we have
\begin{align}
&\sum_{I\sqcup J=[n]\setminus\{k\}}(-1)^{|J\cap X_k|}V_{I,J}=0
 &&\text{if\/ $X_k\ne \emptyset$},\label{eq_cod2_rel_S1}\\
&\sum_{I\sqcup J=[n]\setminus\{k\}}V_{I,J}=\sigma_{n-2} &&\text{if\/ $X_k=\emptyset$.}
\label{eq_cod2_rel_S2}
\end{align} 
\end{theorem}
\begin{proof}
Consider the flat position~$P_0$ of~$P_u$, and apply Lemmas~\ref{lem_rel_basic} and~\ref{lem_rel_basic_S_strong} to the $(n-1)$-dimensional cross-polytopes~$P_{(k)}$ defined in Section~\ref{section_flat}. Facets of $P_{(k)}$ are  $(n-2)$-dim\-en\-sional faces~$\Delta_{I,J}$ of~$P_0$ such that $I\sqcup J=[n]\setminus k$. By Theorem~\ref{theorem_opisan}, either the cross-polytope~$P_{(k)}$ is circumscribed about a hypersurface that is a sphere or an orisphere or an equidistant hypersurface, or all facets of~$P_{(k)}$ are parallel to some hyperplane, or  all facets of~$P_{(k)}$ intersect some hyperplane by the same angle. Besides, the class to which a facet~$\Delta_{I,J}$ belongs is determined by the evenness of the cardinality of the set~$J\setminus X_k$. On the other hand, the colour of the facet $\Delta_{I,J}$ is determined by the evenness of the cardinality of the set~$J$. Hence, for flexible cross-polytopes in~$\E^n$ and~$\Lambda^n$, Lemma~\ref{lem_rel_basic} immediately yield formula~\eqref{eq_cod2_rel_EL}.  

Now, let $P_u$ be a flexible cross-polytope of the simplest type in~$\bS^n$. Consider the  $2^{2n}$ spherical flexible cross-polytopes~$P^{\bs}_{u}$ obtained from~$P_u$ by replacing some of its vertices by the antipodal points, i.\,e., corresponding to the same pair~$(G,\blambda)$ and all rows of signs $\bs=(s_1,\ldots,s_n,s_1',\ldots,s_n').$
For each row of signs~$\bs$, we put
$$
\varepsilon_k(\bs)=s_1\cdots\hat{s}_k\cdots s_n=s_k(s_1\cdots s_n).
$$
The sets~$X_k$, the faces~$\Delta_{I,J}$, and the volumes~$V_{I,J}$ depend on~$\bs$. We shall denote them by~$X_k(\bs)$, $\Delta_{I,J}(\bs)$, and~$V_{I,J}(\bs)$ respectively. 
We denote by~$S_k(\bs)$ the left-hand side of~\eqref{eq_cod2_rel_S1} or~\eqref{eq_cod2_rel_S2} for the flexible cross-polytope~$P^{\bs}_u$. 

For each $u_0\ne 0,\infty$, at least one of the cross-polytopes~$P^{\bs}_{u_0}$ is contained in the closed hemisphere bounded by the $(n-2)$-dimensional great sphere concentric  with the small spheres $\Omega_1,\ldots,\Omega_n$. Therefore, completely in the same way as we have deduced formula~\eqref{eq_cod2_rel_EL} from Lemma~\ref{lem_rel_basic}, we can deduce from Lemma~\ref{lem_rel_basic_S_strong} that there exists a row of signs $\bs^*$ such that $S_k(\bs^*)=0$ for all $k=1,\ldots,n$. Notice that all sets $X_k(\bs^*)$ are non-empty, $k=1,\ldots,n$. Indeed, if a set $X_k(\bs^*)$ was empty, then the corresponding sum~$S_k(\bs^*)$ would be a sum of positive summands, therefore, would be non-zero.

Let us prove formulae~\eqref{eq_cod2_rel_S1} and~\eqref{eq_cod2_rel_S2} by the induction on the dimension~$n$. Obviously, for $n=2$ these formulae hold true, since the volume of any zero-dimensional face is equal to~$1$, and $\sigma_0=2$. Assume that formulae~\eqref{eq_cod2_rel_S1} and~\eqref{eq_cod2_rel_S2}  hold true for all $(n-1)$-dimensional spherical flexible cross-polytopes of the simplest type, and prove them for an  $n$-dimensional spherical flexible cross-polytope of the simplest type~$P_u$. Let us study how the value~$S_k(\bs)$ changes when we change the row of signs~$\bs$. We put,
$$
\overline{S}_k(\bs)=\left\{
\begin{aligned}
&\varepsilon_k(\bs)S_k(\bs)&&\text{if $X_k(\bs)\ne\emptyset$,}\\
&\varepsilon_k(\bs)(S_k(\bs)-\sigma_{n-2})&&\text{if $X_k(\bs)=\emptyset$.}
\end{aligned}
\right.
$$

\begin{lem}\label{lem_n-2}
Take any\/~$k\in[n]$. Then the numbers\/~$\overline{S}_k(\bs)$ are the same for all rows of signs~$\bs$.
\end{lem}

\begin{proof}
It is sufficient to prove that $\overline{S}_k\bigl(\bs^{(1)}\bigr)=\overline{S}_k\bigl(\bs^{(2)}\bigr)$ for any two rows of signs~$\bs^{(1)}$ and~$\bs^{(2)}$ that differ by one coordinate only. Let this coordinate be either~$s_l$ or~$s_l'$.

We put~$P^{(i)}_u=P^{\bs^{(i)}}_u$, $X^{(i)}_k=X_k\bigl(\bs^{(i)}\bigr)$, $\Delta_{I,J}^{(i)}=\Delta_{I,J}\bigl(\bs^{(i)}\bigr)$, $S^{(i)}_k=S_k\bigl(\bs^{(i)}\bigr)$, etc., $i=1,2$. The vertices of~$P^{(1)}_u$ will be denoted by~$\ba_j$ and~$\bb_j$ instead of~$\ba_j^{(1)}$ and~$\bb_j^{(1)}$ respectively. The cross-polytope~$P^{(2)}_u$ is obtained from the cross-polytope~$P^{(1)}_u$ by replacing one of its vertices~$\bv$ by the antipodal point of the sphere. We have,  $\bv=\ba_l$ if the rows~$\bs^{(1)}$ and~$\bs^{(2)}$ differ by the coordinate~$s_l$, and $\bv=\bb_l$ if the rows~$\bs^{(1)}$ and~$\bs^{(2)}$ differ by the coordinate~$s_l'$. 

If $l=k$, then $P_{(k)}^{(1)}=P_{(k)}^{(2)}$, hence, $S_k^{(1)}=S_k^{(2)}$. Besides, $X_k^{(1)}=X_k^{(2)}$ and $\varepsilon_k^{(1)}=\varepsilon_k^{(2)}$. Therefore, equality $\overline{S}_k^{(1)}=\overline{S}_k^{(2)}$ holds true. 
 
Suppose that $l\ne k$. It is easy to see that the sets~$X^{(1)}_k$ and~$X^{(2)}_k$ differ by the occurrence of the element~$l$ only, that is, either $l\notin X^{(1)}_k$ and $X^{(2)}_k=X^{(1)}_k\cup\{l\}$ or $l\notin X^{(2)}_k$ and $X^{(1)}_k=X^{(2)}_k\cup\{l\}$. We assume that $l\notin X^{(1)}_k$ and $X^{(2)}_k=X^{(1)}_k\cup\{l\}$, since the second case is completely similar.

Let $\widetilde{P}_u=L\bigl(\bv,P^{(1)}_u\bigr)$ be the link of the vertex~$\bv$ in the flexible cross-polytope~$P^{(1)}_u$. As it was shown in the introduction, the link of a vertex of a flexible polyhedron is itself a spherical flexible polyhedron. Besides, the dihedral angles of the link are equal to the corresponding dihedral angles of the initial polyhedron. Since $P^{(1)}_u$ is a cross-polytope of the simplest type, its dihedral angles vary so that the tangents of their halves are either directly or inversely proportional to each other. Hence the same property holds true for the dihedral angles of the spherical flexible cross-polytope $\widetilde{P}_u$, which implies that $\widetilde{P}_u$ is also a flexible polyhedron of the simplest type, see Remark~\ref{rem_dihedral}. We denote by~$\widetilde{\ba}_j$ and~$\widetilde{\bb}_j$ the vertices of~$\widetilde{P}_u$ that are the unit tangent vectors to the edges~$[\bv\ba_j]$ and~$[\bv\bb_j]$ respectively, $j\ne l$.  For the cross-polytope~$\widetilde{P}_u$ all objects introduced above will be marked off by a tilde, for example, $\widetilde{G}$, $\widetilde{\blambda}$, $\widetilde{\bs}$,  $\widetilde{\Delta}_{I,J}$, $\widetilde{S}_{j}$, etc. Notice that  pairs of opposite vertices of the cross-polytope~$\widetilde{P}_u$ are indexed by elements of the set~$[n]\setminus\{l\}$ rather than $[n-1]$. This is completely inessential. The only difference is that the set~$\widetilde{X}_l$ and the sets~$I$ and~$J$ for faces $\widetilde{\Delta}_{I,J}$ are also subsets of~$[n]\setminus\{l\}$. It is easy to see that~$\widetilde\blambda$ is the row~$\blambda$ with the coordinate~$\lambda_l$ deleted. Hence, the entries of the row~$\widetilde{\blambda}$ are positive and increase. Therefore, by the inductive assumption, formulae~\eqref{eq_cod2_rel_S1},~\eqref{eq_cod2_rel_S2} hold true for the cross-polytope~$\widetilde{P}_u$.   

For $u=0$, all dihedral angles of the cross-polytope~$P^{(1)}_u$  at facets~$\Delta_{I,J}^{(1)}$ such that $I\sqcup J=[n]\setminus\{j\}$ degenerates to the zero angles if $s_j^{(1)}{s_j'}^{(1)}=1$, and to the straight angles if $s_j^{(1)}{s_j'}^{(1)}=-1$. The same is true for the dihedral angles of the $(n-1)$-dimensional cross-polytope~$\widetilde{P}_u$. Besides, the dihedral angle of~$\widetilde{P}_u$ at the $(n-3)$-dimensional face~$\widetilde{F}$ cut by  the tangent cone to the $(n-2)$-dimensional face~$F$ of~$P^{(1)}_u$ is equal to the dihedral angle of~$P^{(1)}_u$ at~$F$. Therefore, $\widetilde{s}_j\widetilde{s}_j^{\,\prime}=s_j^{(1)}{s_j'}^{(1)}$ for all~$j\ne l$. Hence, $\widetilde{X}_k=X^{(1)}_k$.

If an $(n-2)$-dimensional face~$\Delta_{I,J}^{(1)}$ does not contain the vertex~$\bv$, then  $\Delta_{I,J}^{(1)}=\Delta_{I,J}^{(2)}$, hence, $V_{I,J}^{(1)}=V_{I,J}^{(2)}$. If an $(n-2)$-dimensional face~$\Delta_{I,J}^{(1)}$ contains the vertex~$\bv$, then the faces~$\Delta_{I,J}^{(1)}$ and~$\Delta_{I,J}^{(2)}$ constitute together an $(n-2)$-dimensional spherical $(n-2)$-hedron with the two antipodal vertices~$\bv$ and~$-\bv$, and the intersection of this $(n-2)$-hedron with the equatorial great sphere with respect to the poles~$\bv$ and~$-\bv$ is isometric to the simplex~$\widetilde{\Delta}_{I,J}$, see Fig.~\ref{fig_dolka}. Hence, 
$$
V_{I,J}^{(1)}+V_{I,J}^{(2)}=\frac{\sigma_{n-2}}{\sigma_{n-3}}\,\widetilde{V}_{I,J}.
$$
If $\bv=\ba_l$, then the signs $(-1)^{|J\cap X_k^{(1)}|}$ and $(-1)^{|J\cap X_k^{(2)}|}$ in the sums~$S_k^{(1)}$ and~$S_k^{(2)}$ respectively coincide to each other whenever $\bv\in\Delta_{I,J}$, and are opposite to each other whenever $\bv\notin\Delta_{I,J}$. If $\bv=\bb_l$, then, vice versa,  the signs $(-1)^{|J\cap X_k^{(1)}|}$ and $(-1)^{|J\cap X_k^{(2)}|}$  are opposite to each other whenever $\bv\in\Delta_{I,J}$, and coincide to each other whenever $\bv\notin\Delta_{I,J}$. Hence, 
$$
S_k^{(1)}=\pm S_k^{(2)} +\frac{\sigma_{n-2}}{\sigma_{n-3}}\,\widetilde{S}_k,
$$
where we choose the sign~$-$ if $\bv=\ba_l$, and the sign~$+$ if $\bv=\bb_l$. By the inductive assumption, $\widetilde{S}_k=\sigma_{n-3}$ if $\widetilde{X}_k=\emptyset$, and $\widetilde{S}_k=0$ if $\widetilde{X}_k\ne\emptyset$. Since $\widetilde{X}_k=X_k^{(1)}$ and the set~$X_k^{(2)}$ is non-empty, this immediately implies the required equality $\overline{S}_k^{(1)}=\overline{S}_k^{(2)}$.
\end{proof}

\begin{figure}
\begin{center}
\includegraphics[scale=.85]{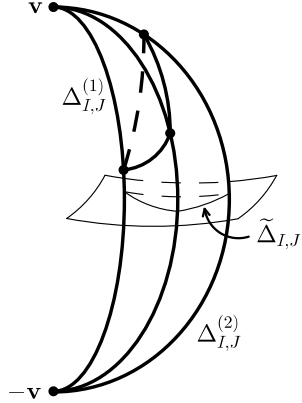}
\end{center}
\caption{The union of the faces~$\Delta_{I,J}^{(1)}$ and~$\Delta_{I,J}^{(2)}$}\label{fig_dolka}
\end{figure}

Now, we are ready to complete the proof of Theorem~\ref{theorem_n-2}. Since $S_k(\bs^*)=0$ and $X_k(\bs^*)\ne\emptyset$ for $k=1,\ldots,n$, Lemma~\ref{lem_n-2} implies that $\overline{S}_k(\bs)=0$ for all~$k$ and~$\bs$. Hence, $S_k(\bs)=0$ whenever $X_k(\bs)\ne\emptyset$, and $S_k(\bs)=\sigma_{n-2}$ whenever $X_k(\bs)=\emptyset$. Therefore, the assertion of Theorem~\ref{theorem_n-2} is true for all cross-polytopes~$P^{\bs}_u$, in particular, for the initial cross-polytope~$P_u$.
\end{proof} 
 
\subsection{Volumes of flexible cross-polytopes of the simplest type} First of all, let us give a rigorous  definition of the volume of a not necessarily embedded polyhedron $P\colon K\to \X^n$.

Suppose that $\X^n=\E^n$ or~$\Lambda^n$. For each point~$\bx\in\X^n\setminus P(K)$, we denote by $\varkappa(\bx)$ the algebraic intersection index of a continuous curve~$\gamma$ from~$\bx$ to the infinity with the oriented cycle~$P(K)$. Obviously, this intersection index is independent of the choice of the curve~$\gamma$.  Then $\varkappa$ is an almost everywhere defined integral-valued piecewise constant function on~$\X^m$ with compact support.  By definition, the \textit{generalized oriented volume\/} of the polyhedron~$P$ is the number
\begin{equation}\label{eq_OV}
\CV(P)=\int_{\X^n} \varkappa(\bx)\,\mathrm{d}V(\bx),
\end{equation}
where $\mathrm{d}V$ is the standard volume form on~$\X^n$.
 
The case $\X^n=\bS^n$ is somewhat more complicated, since the sphere~$\bS^n$ does not contain the infinity, which implies that there is no canonical choice of the function~$\varkappa$. Consider an arbitrary integral-valued piecewise constant function~$\varkappa$ defined on~$\bS^m\setminus P(K)$ such that, for each two points~$\bx$ and~$\by$, the difference $\varkappa(\bx)-\varkappa(\by)$ is equal to the algebraic intersection index of a curve~$\gamma$ from~$\bx$  to~$\by$ with the oriented cycle~$P(K)$. Such function exists an is unique up to the addition of an integral constant. Hence formula~\eqref{eq_OV} determines the number~$\CV(P)$ up to the addition of a multiple of the volume~$\sigma_n$ of the sphere~$\bS^n$.  The obtained well-defined element $\CV(P)\in\R/(\sigma_n\Z)$ will be called the \textit{generalized oriented volume\/} of the spherical polyhedron~$P$. 

Now, let $P_u$ be a flexible cross-polytope of the simplest type in~$\X^n$ corresponding to the set of data $(G,\blambda,\bs)$. As before, we assume that $0<\lambda_1<\cdots<\lambda_n$.

\begin{theorem}\label{theorem_vol} The generalized volume of any flexible cross-polytope of the simplest type in the spaces\/~$\E^n$ and\/~$\Lambda^n,$ $n\ge 3,$ is identically equal to zero during the flexion. The generalized volume of any flexible cross-polytope of the simplest type in the sphere\/~$\bS^n,$ $n\ge 2,$ is identically equal to zero during the flexion, except for the following cases:
\begin{itemize}
\item If\/ $s_1s_1'=\cdots=s_ns_n'=1,$ then
\begin{equation}\label{eq_vol1}
\CV(P_u)=\frac{s_1\sigma_n}{\pi}\arctan(\lambda_1u).
\end{equation}
\item If\/ $s_1s_1'=\cdots=s_ns_n'=-1$, then
\begin{equation}\label{eq_vol2}
\CV(P_u)=\frac{\sigma_n}{2}+\frac{s_n\sigma_n}{\pi}\arctan(\lambda_nu).
\end{equation}
\item If\/ $s_1s_1'=\cdots=s_ks_k'=-1,$ $s_{k+1}s'_{k+1}=\cdots=s_{n}s'_{n}=1$ for some $k,$ $1\le k<n,$ then
\begin{equation}\label{eq_vol3}
\CV(P_u)=\frac{\sigma_n}{\pi}\bigl(s_k\arctan(\lambda_ku)+s_{k+1}\arctan(\lambda_{k+1}u)\bigr).
\end{equation}
\end{itemize}
\end{theorem}

\begin{proof}
In the Euclidean space~$\E^n$, the volume of any flexible polyhedron is constant during the flexion, see~\cite{Sab96}--\cite{Sab98b}, \cite{CSW97} for $n=3$, \cite{Gai11} for $n=4$, and~\cite{Gai12} for $n\ge 5$. Hence, $\CV(P_u)=\CV(P_0)=0$. The latter equality is true, since the cross-polytope~$P_0$ is flat, i.\,e., is contained in the hyperplane~$\E^{n-1}$.

In the two-dimensional sphere $\bS^2$, flexible cross-polytopes of the simplest type are exactly flexible quadrangles shown in Fig.~\ref{fig_3polyg}. For each of them the assertion of the theorem can be checked immediately.

Suppose that $\X^n=\Lambda^n$ or~$\bS^n$, $n\ge 3$. We introduce the parameter~$\varepsilon$ that is equal to~$1$ for~$\bS^n$ and to~$-1$ for~$\Lambda^n$. Schl\"afli's formula is the following classical formula for the differential of the oriented volume of an arbitrary polyhedron in~$\X^n$ that is deformed preserving its combinatorial type:
$$
\mathrm{d}\CV=\frac{\varepsilon}{n-1}\sum_FV_{n-2}(F)\,\mathrm{d}\psi_F,
$$
where the sum is taken over all  $(n-2)$-dimensional faces~$F$, and $\psi_F$ is the oriented dihedral angle of the polyhedron at the face~$F$. In our case, this formula takes the form
$$
\mathrm{d}\CV(P_u)=\frac{\varepsilon}{n-1}\sum_{(I,J)}V_{I,J}\,\mathrm{d}\psi_{I,J}(u),
$$
where the sum is taken over all pairs of non-intersecting subsets $I,J\subset[n]$ such that $|I|+|J|=n-1$.

\begin{remark}
Schl\"afli's formula is usually written for convex polytopes, see, for instance,~\cite{AVS88}. Schl\"afli's formula for arbitrary non-degenerate polyhedra in the sense of our definition in Section~\ref{section_defin} follows from Schl\"afli's formula for simplices, since the indicator function~$\varkappa_P(\bx)$ of any $n$-dimensional polyhedron~$P\colon K\to\X^n$ can be represented as an algebraic sum of the indicator functions of simplices with vertices at vertices of~$P$. Actually, if $v_0$ is an arbitrary vertex of the pseudo-manifold~$K$, and $\bigl[v_1^{(j)}\ldots v_n^{(j)}\bigr]$, $j=1,\ldots,N$,  are all positively oriented $(n-1)$-dimensional simplices of~$K$ that do not contain the vertex~$v_0$, then
$$
\varkappa_P(\bx)=\sum_{j=1}^N\eta_j\varkappa_{\bigl[P\bigl(v_1^{(j)}\bigr)\ldots P\bigl(v_n^{(j)}\bigr)\bigr]}(\bx),
$$
where $\eta_j$ is the sign of the orientation of the simplex $\bigl[P\bigl(v_1^{(j)}\bigr)\ldots P\bigl(v_n^{(j)}\bigr)\bigr]$.
\end{remark}

If $\X^n=\Lambda^n$, then, using fomulae~\eqref{eq_dihedral} and~\eqref{eq_cod2_rel_EL}, we obtain that $\mathrm{d}\CV(P_u)=0$. Hence, $\CV(P_u)=\CV(P_0)=0$.

Similarly, if $\X^n=\bS^n$, then, using fomulae~\eqref{eq_dihedral}, \eqref{eq_cod2_rel_S1}, and~\eqref{eq_cod2_rel_S2}, we obtain that
$$
\mathrm{d}\CV(P_u)=\frac{2\sigma_{n-2}}{n-1}\sum_{k\colon X_k=\emptyset}s_k\,\mathrm{d}\arctan(\lambda_ku).
$$
Since $\sigma_n=2\pi\sigma_{n-2}/(n-1)$, this can be rewritten in the form
$$
\CV(P_u)=\CV(P_0)+\frac{\sigma_{n}}{\pi}\sum_{k\colon X_k=\emptyset}s_k\arctan(\lambda_ku).
$$

The boundary of the cross-polytope~$P_0$ is contained in the great sphere~$\bS^{n-1}$. Hence, the oriented volume~$\CV(P_0)$ is equal to~$0$ if $\deg P_0=0$, and is equal to~$\sigma_n/2$ if $\deg P_0= 1$. (In the latter case, the sign is irrelevant, since the volume is defined up to a multiple of~$\sigma_n$.) By Lemma~\ref{lem_deg}, we obtain that $\CV(P_0)=\sigma_n/2$ if $s_is'_i=-1$ for all~$i$, and $\CV(P_0)=0$ in all other cases.

It is checked immediately that all sets~$X_k$ are non-empty always except for the three special cases listed in Theorem~\ref{theorem_vol}. Hence the oriented volume of~$P_u$ is identically equal to zero always except for these three cases. If $s_1s_1'=\cdots=s_ns_n'=1$, then $X_1$ is the only empty set among the sets~$X_k$, and we obtain~\eqref{eq_vol1}. If $s_1s_1'=\cdots=s_ns_n'=-1$, then $X_n$ is the only empty set among the sets~$X_k$, and we obtain~\eqref{eq_vol2}. If $s_1s_1'=\cdots=s_ks_k'=-1$ и $s_{k+1}s_{k+1}'=\cdots=s_ns_n'=1$ for a~$k$ such that $1\le k<n$, then $X_k=X_{k+1}=\emptyset$, and all other sets~$X_j$ are non-empty. Hence, we obtain formula~\eqref{eq_vol3}.
\end{proof}

\begin{remark}
In the Euclidean case, Schl\"afli's formula does not allow to obtain any expression for the volume of a polyhedron, since the volume does not enter this formula.  Instead, in the Euclidean case, Schl\"afli's formula implies that the total mean curvature of a flexible polyhedron is constant during the flexion, see~\cite{Ale85}.
\end{remark}

\begin{cor}
The Modified Bellows Conjecture \textnormal{(}Conjecture~\ref{con_mod}\textnormal{)} is true for all spherical flexible cross-polytopes of the simplest type.
\end{cor}

\begin{proof}
We can always replace by their antipodes several vertices of a spherical flexible cross-polytope of the simplest type~$P_u$ so that to obtain that the signs~$s_j$, $s_j'$ corresponding to the obtained cross-polytope satisfy~$s_1s_1'=1$, $s_2s_2'=-1$. Then, by Theorem~\ref{theorem_vol}, the generalized oriented volume of the obtained flexible cross-polytope is identically equal to zero during the flexion.
\end{proof}

The author is grateful to the referee for a series of useful comments.

\end{document}